\documentclass[12pt]{amsart}
\usepackage{times}
\usepackage{anysize}
\marginsize{2.7cm}{2.44cm}{2.0cm}{3.0cm}
 \usepackage[dvips]{graphicx}
 \usepackage{wrapfig}
 \usepackage{amsmath}
 \usepackage{amsthm}
 \usepackage{amsfonts}
 \usepackage{amssymb}
 \usepackage{layout}
 \usepackage{verbatim}
 \usepackage{alltt}
\usepackage{yfonts}
\usepackage[T1]{fontenc}

\usepackage[all]{xy}
\usepackage{setspace}
\newtheorem*{ithm}{Theorem}
\newtheorem*{thma}{Theorem A}
\newtheorem*{thmb}{Theorem B}

\newtheorem*{claim}{Claim}

\newcommand{\LL}{\Lambda}
\newcommand{\QQ}{\mathbb{Q}}
\newcommand{\FF}{\mathcal{F}}

\newcommand{\lra}{\longrightarrow}
\newcommand{\ZZ}{\mathbb{Z}}

\newcommand{\oo}{\mathcal{O}}

\newcommand{\ra}{\rightarrow}

\newcommand{\be}{\begin{equation}}
\newcommand{\ee}{\end{equation}}

\newcommand{\al}{\mathcal{L}}

\newcommand{\ff}{\hbox{\frakfamily f}}

\newcommand{\FFc}{\mathcal{F}_{\textup{\lowercase{can}}}}

\numberwithin{equation}{section}
\newtheorem{thm}{Theorem}[section]
\newtheorem{lemma}[thm]{Lemma}
\newenvironment{define}{\par\medskip\noindent\refstepcounter{thm}
\bgroup{\hspace*{-0.15 cm}\bf{Definition}
\thethm.}\bgroup}{\egroup \egroup\par\medskip}\newtheorem{prop}[thm]{Proposition}
\newtheorem{cor}[thm]{Corollary}
\newenvironment{rem}{\par\medskip\noindent\refstepcounter{thm}
\bgroup{\hspace*{-0.15 cm}\bf{Remark} \thethm.}\bgroup}{\egroup
\egroup\par\medskip} \parskip 2pt

\begin{document}
\title{{H}\lowercase{eight pairings,} {E}\lowercase{xceptional zeros and} {R}\lowercase{ubin's formula}: {T}\lowercase{he} {M}\lowercase{ultiplicative group}}

\author{K\^az\i m B\"uy\"ukboduk}

\address{Kazim Buyukboduk  \hfill\break\indent
Max Planck Institut f\"ur Mathematik
 \hfill\break\indent Vivatsgasse 7,
Bonn 53111
\hfill\break\indent DEUTSCHLAND} \email{{\tt kazim@mpim-bonn.mpg.de}\hfill\break\indent {\it
Web page:} {\tt http://guests.mpim-bonn.mpg.de/$\sim$kazim}}
\curraddr{\hfill\break\indent Ko\c{c} University,  Mathematics \hfill\break\indent Rumeli Feneri Yolu \hfill\break\indent 34450 Sar\i yer/\.Istanbul
\hfill\break\indent TURKEY}
\keywords{Exceptional Zeros, Cyclotomic Units, Height Pairings, $p$-adic $L$-functions}
\subjclass[2000]{Primary 11R23, 11R34, 11R42; Secondary 11G50}

\begin{abstract}
In this paper we prove a formula, much in the spirit of one due to Rubin, which expresses the leading coefficients of various $p$-adic $L$-functions in the presence of an exceptional zero in terms of Nekov\'{a}\v{r}'s $p$-adic height pairings on his extended Selmer groups. In a particular case, the Rubin-style formula we prove recovers a $p$-adic Kronecker limit formula. 
In a disjoint case, we observe that our computations with Nekov\'a\v{r}'s heights agree with the Ferrero-Greenberg formula (more generally, Gross' conjectural formula) for the leading coefficient of the Kubota-Leopoldt $p$-adic $L$-function (resp., the Deligne-Ribet $p$-adic $L$-function) at $s=0$.
\end{abstract}

\maketitle
\tableofcontents
\section{Introduction}

The celebrated formula of Gross and Zagier~\cite{gz} expresses the first derivative at $s=1$ of a Rankin $L$-series of a modular form $f$ of weight 2 on $\Gamma_0(N)$ in terms of the N\'eron-Tate height of a Heegner point on the $f$-quotient $A_f$ of the Jacobian $J_0(N)$ of the modular curve $X_0(N)$. A $p$-adic variant of this formula has been proved by Perrin Riou~\cite{pr87}, relating the $p$-adic height of a Heegner point on $A_f$ to a first derivative (taken in the cyclotomic direction) of a \emph{two-variable} $p$-adic $L$-function associated to $f$. (See also \cite{howard:GZ} for a generalization of this formula with more Iwasawa theoretical flavor). Later, Nekov\'{a}\v{r}~\cite{nek95} extended the results of~\cite{pr87} to higher weight modular forms, where he utilized his $p$-adic heights defined earlier in~\cite{nek2}.

When $E$ is an elliptic curve defined over $\QQ$ with CM and $p$ is an odd prime at which $E$ has good, ordinary reduction, Perrin-Riou~\cite{pr83} gives a purely algebraic construction of the canonical $p$-adic height pairing on the $p$-adic Selmer group $S_p(E/\QQ)$. If further  $L(E/\QQ,1)=0$, Rubin~\cite{ru:rational92} obtains a formula for the special values of the associated Katz two-variable $p$-adic $L$-function in terms of  the $p$-adic height of an element $x_p \in S_p(E)$ (which is constructed from elliptic units). When $E$ does not have CM, but still good, ordinary at $p$, results along this line have been obtained by Perrin-Riou~\cite{pr93} utilizing Nekov\'a{\v r}'s definition of $p$-adic heights~\cite{nek2} and Kato's zeta-elements~\cite{kato}. Perrin-Riou's formula in~\cite{pr93} goes hand-in-hand with Rubin's result~\cite[Theorem 1]{ru94} (which follows from Theorem 3.2 of {loc.cit.}; this is the version of \emph{Rubin's formula} we refer to in the abstract). Rubin uses in~\cite{ru94} the definition of~\cite{pr:ht} for $p$-adic height pairings. We finally note that Rubin's formula \cite[Theorem 3.2]{ru94} has been generalized by Howard \cite[Theorem 3.4]{howard:heights} for abelian varieties (resp., by Nekov\'a\v{r}~\cite[\S11.5.10]{nek} for general motives) whose $L$-functions vanish to higher order. We provide an overview of Rubin's formula  since it is one of the main motivations for the results of the current paper.       

Suppose $E_{/\QQ}$ is an elliptic curve which has good, ordinary reduction at $p$. Let $\QQ_\infty$ be the unique $\ZZ_p$-extension of $\QQ$, and for every $n$, let $\QQ_n$ be the unique sub-extension of $\QQ$ of degree $p^n$. Put $\Phi_n=\QQ_n\otimes\QQ_p$ and $\Phi_\infty=\cup\Phi_n$. Let $T_p(E)$ denote the $p$-adic Tate module of $E$, and suppose we are given a sequence of cohomology classes $\mathbf{z}=\{z_n\}\in \varprojlim H^1(\QQ_n,T_p(E))$. Using local Tate cup-product pairing, one obtains an element $f_{\mathbf{z}}\in \textup{Hom}(E(\Phi_\infty),\ZZ_p)$; see equation (5) of~\cite{ru94}. The following is Theorem 3.2(i) of {loc.cit.}:
\begin{ithm}[Rubin]
Let $S_{p}(E/\QQ)$ denote the $p$-adic Selmer group of $E_{\QQ}$ over $\QQ$. Then $z_0\in S_{p}(E/\QQ)$ if and only if $f_{\mathbf{z}}(E(\QQ_p))=0$.       
\end{ithm}

When $f_{\mathbf{z}}(E(\QQ_p))=0$, Rubin constructs in \cite[\S3]{ru94} a \emph{derivative} $\textup{Der}_\rho(f_{\mathbf{z}})$  of $f_{\mathbf{z}}$ along $\rho$, 
where $\rho$ is any nonzero homomorphism $\textup{Gal}(\QQ_\infty/\QQ)\lra\ZZ_p$. See also the remarks preceding Theorem 3.2 and Proposition 7.1 of \cite{ru94}. Rubin's formula can be stated as follows:
\begin{ithm}[Rubin] Suppose $z_0 \in S_p(E/\QQ)\subset H^1(\QQ,T_p(E))$. Then for every $x \in E(\QQ)\otimes\ZZ_p$,
 $$\langle z_0,x\rangle_\rho=\textup{Der}_\rho(f_{\mathbf{z}})(x),$$
  where $\langle\,,\,\rangle_\rho$ is the $p$-adic height pairing.
\end{ithm}

This formula should be compared to our formula stated in Theorem~\ref{main-primitive}. Having spelled out the first link between our work and results mentioned above, let us describe our results in greater detail.

In~\cite{nek}, Nekov\'{a}\v{r} defines extended Selmer groups associated to (ordinary) Galois representations, which are strictly larger than the classical Selmer groups in the presence of an exceptional zero (in the sense of~\cite{gr}). He also defines $p$-adic height pairings on his extended Selmer groups. One natural question is what portion of the results above may be transferred to this new setting when an exceptional zero is present. We tackle this problem in the simplest and the most classical setting: Fixing a number field $K$, the Galois representation in consideration is $T=\oo(1)\otimes\chi^{-1}$. Here, $\oo$ is the ring of integers of a finite extension $\frak{F}$ of $\QQ_p$ and $\oo(1)=\oo\otimes_{\ZZ_p}\ZZ_p(1)$, where $\ZZ_p(1)=T_p(\mathbb{G}_m)$ is as usual the $p$-adic Tate module of the multiplicative group, and $\chi:\textup{Gal}(\overline{K}/K)\ra\oo^\times$ is a non-trivial Dirichlet character with the property that $\chi(\wp)=1$ for a prime $\wp$ of $K$ lying above $p$. The Rubin-style formula we prove here (Corollary~\ref{main}) is akin to \cite[Theorem~1]{ru94}. Before we state it, we introduce the necessary notation.

Suppose in this introduction that $K=\QQ$ and $\chi$ is an even Dirichlet character. See \S\ref{subsec:htQodd} below for the case when $K$ is a general totally real number field but $\chi$ is totally odd, and \S\ref{subsec:caseimagquad} when the base field $K$ is totally imaginary. Let $L$ be the field cut by the Dirichlet character $\chi$, i.e., the fixed field of $\ker(\chi)$. Let $c_1^\chi \in \widetilde{H}^1_f(\QQ,T)$ be \emph{tame} cyclotomic unit inside of $L$ defined as in~\cite[\S6.1]{mr02}, see also \S\ref{sec:cyclo} below for a recap. Here (and below) $\widetilde{H}^1_f(K,T)$ stands for the extended Selmer groups of Nekov\'{a}\v{r}; for an overview (and explicit calculations specific to our case of interest, including a description of how we view the cyclotomic units as elements of the extended Selmer groups) see~\S\ref{subsec:tilde} and \S\ref{sec:cyclo} below. Set $T^*=\textup{Hom}(T,\oo(1))=\oo(\chi)$. Let $\langle\,,\,\rangle_{\textup{Nek}}$ denote Nekov\'{a}\v{r}'s $p$-adic height pairing, see~\cite[\S11]{nek} for a general definition, and also~\S\ref{sec:hts} below for the portion of the theory that concerns us. Attached to an arbitrary element $\alpha \in \widetilde{H}^1_f(\QQ,T^*)$ and the collection of cyclotomic units $\xi$ along the cyclotomic $\ZZ_p$-tower, we construct a `$p$-adic $L$-function' $L_{\xi,\Phi}$ in \S\ref{sec:main} below. The Rubin-style formula we prove reads as follows:

\begin{thma}[\textup{Corollary~\ref{main}} below]\,\,$\langle c_1^\chi, \alpha\rangle_{\textup{Nek}}=L^{\prime}_{\xi,\Phi}(\pmb{1}).$
\end{thma}
Here, $\pmb{1}$ is the trivial character and $L_{\xi,\Phi}^{\prime}$ is the derivative of $L_{\xi,\Phi}$ \emph{along the cyclotomic character}, see \S\ref{sec:main} for details. Using Coleman's map, one may choose a particular $\Phi$ and $\alpha$, and apply Theorem A above to prove:
\begin{thmb}[\textup{Theorem~\ref{main:K-L}} below]
$\langle c_1^\chi, \frak{col}_0^\chi \rangle_{\textup{Nek}}=\widetilde{L}_p^\prime(1,\chi).$
\end{thmb}
Here, $\frak{col}_0^\chi \in \widetilde{H}^1_f(\QQ,T^*)$ is the element we obtain from Coleman's homomorphism and $\widetilde{L}_p(s,\chi)$ is an \emph{imprimitive} Kubota-Leopoldt $p$-adic $L$-function. See \S\ref{subsec:htQeven} for details. See also  \S\ref{subsec:caseimagquad}  for the version of this result when the base field is a quadratic imaginary number field. We remark that  our formula above recovers a $p$-adic variant of Kronecker's limit formula with a new perspective offered by Nekov\'a\v{r}'s theory.

In \S\ref{subsec:htQodd}, we present similar results for totally odd characters $\chi$ (when the base field K is totally real). We remark for now that when $K$ is an arbitrary totally real number field and $\chi$ is totally odd, our calculations provide a new interpretation for Gross' conjecture (and for the Ferrero-Greenberg theorem when $k=\QQ$). See Theorem~\ref{thm:htswhenodd} and Remark~\ref{rem:totallyrealDDP} below.

See also Remark~\ref{rem:ellipticcurvesKurihara} for a related observation when the Galois representation in question is the $p$-adic Tate-module of an elliptic curve $E_{/\QQ}$ which has split-multiplicative reduction at $p$. 
 
The layout of the paper is as follows: In Section~\ref{sec:hts} we give an overview of Nekov\'{a}\v{r}'s theory of Selmer complexes and $p$-adic height pairings. We explicitly describe these objects in \S\ref{subsec:classical} in the cases of interest. In sections~\ref{sec:cyclo}-\ref{sec:main} we restrict our attention to the case $K=\QQ$ and $\chi$ even, and to the case when the base field $K$ is totally imaginary. In Section~\ref{sec:cyclo}, we define three types of cyclotomic ($p$-) units which our calculations rely on. In Section~\ref{sec:calculations}, we calculate the $p$-adic height pairing on these different types of cyclotomic ``units'', and use our computations in Section~\ref{sec:main} to prove a Rubin-style formula. In \S\ref{sec:KLNek}, we use this formula to compute the leading coefficients of certain $p$-adic $L$-functions  in terms of Nekov\'a\v{r}'s heights. 
  
We remark that the results of this paper are not covered by Nekov\'{a}\v{r}'s \cite{nek} general treatment (e.g., by his variant of Rubin's formula in \S11.3.15 and \S11.5.10; nor by his calculations in \S11.4.8). In particular, \cite[Remark 11.4.10]{nek} is erroneous. It would be of interest to extend the formalism developed in \cite[\S11.4]{nek} to cover our setting. 

A line of apology: We gave a very detailed and long outline of prior results of `Gross-Zagier type', although the conclusions of the current paper only concern a very particular (and simple) Galois representation. This is mainly because of the author's desire to translate/transform the results in other settings into the context of~\cite{nek}.  

\subsubsection*{Acknowledgements} The author wishes to thank Ralph Greenberg and Tadashi Ochiai for helpful discussions; Karl Rubin for helpful correspondence and David Burns for an informative conversation on the results of this paper. He also is grateful to Masato Kurihara for explaining the author a related result he proved in a different setting. Special thanks are due to Jan Nekov\'{a}\v{r} for his encouragement and for many enlightening discussions. The author started his work on this project  while he was supported by a William Hodge Postdoctoral Fellowship at IH\'ES and the final form of this paper was written up during his stay at Max Planck Institut f\"ur Mathematik. The author thanks both these institutes for their hospitality. The author also thanks the anonymous referee for pointing out several inaccuracies in an earlier version.

\subsection{Notation and Hypotheses}
\label{subsec:notation}
 Fix once and for all a rational prime $p>2$. For a number field $K$, write $G_K$ for the absolute Galois group $\textup{Gal}(\overline{K}/K)$. Let $\oo$ be the ring of integers of a finite extension $\frak{F}$ of $\QQ_p$, and let  $\chi$ denote a non-trivial Dirichlet character 
$$\chi:G_K \lra \oo^\times,$$
 which has prime-to-$p$ order
 \,and which satisfies $\chi(\wp)=1$ for a prime $\wp\subset K$ lying above $p$. In this paper, we will only\footnote{Except in Remark~\ref{rem:totallyrealDDP}, where we say how the arguments of \S\ref{subsec:htQodd} apply for a general totally real number field.}
  deal with the case $K=\QQ$  or $K=k$, where $k$ is a quadratic imaginary number field such that the prime $p$ splits in $k/\QQ$.

Define $T=\oo(1)\otimes\chi^{-1}$ and $T^*=\oo(\chi)$, rank one $\oo$-modules with a $G_K$-action.  Here $\oo(1)$ is the Tate twist.

Let $L$ will be the fixed field of $\ker(\chi)$ and let $\Delta=\textup{Gal}(L/K)$.  Our assumption that $\chi(\wp)=1$ is equivalent to saying that  $\wp$ splits completely in $L/K$. Let $S_\wp=\{v|\wp\}$ denote the collection of places of $L$ above $\wp$ (the letter ``$v$'' is reserved to stand for these places of $L$), and let $L_v$ denote the completion of $L$ at $v$. Although $L_v=K_\wp$ for each $v$, we will distinguish the completions of $L$ at different places (as different embeddings $L\hookrightarrow\overline{\QQ}_p$) and set $G_v=\textup{Gal}(\overline{\QQ}_p/L_v)$ for a fixed algebraic closure $\overline{\QQ}_p$ of $\QQ_p$.

Fix once and for all embeddings $\iota_\infty:\overline{\QQ} \hookrightarrow \mathbb{C}$, and $\iota_p: \overline{\QQ} \hookrightarrow \overline{\QQ}_p$. The choice of $\iota_p$  fixes a prime $v_0 \in S_\wp$.
  
Let $\QQ_\infty/\QQ$ denote the cyclotomic $\ZZ_p$-extension of $\QQ$ and let $\Gamma=\textup{Gal}(\QQ_\infty/\QQ)$. We write $\rho_{\textup{cyc}}$ for the cyclotomic character $\rho_{\textup{cyc}}:\Gamma\stackrel{\sim}{\ra}1+p\ZZ_p$. Let $\QQ_n$ denote the unique sub-extension of $\QQ_\infty/\QQ$ of degree $p^n$ over $\QQ$, i.e., the fixed field of $\Gamma^{p^n}$. Let $\Phi_n$ be the completion of $\QQ_n$ at the unique prime of $\QQ_n$ above $p$, and set $\Phi_\infty=\cup \Phi_n$, the cyclotomic $\ZZ_p$-extension of $\QQ_p$. By slight abuse of notation $\textup{Gal}(\Phi_\infty/\QQ_p)$ will be denoted by $\Gamma$ as well. We fix a topological generator $\gamma$ of $\Gamma$. We also set $\LL=\oo[[\Gamma]]$ as the cyclotomic Iwasawa algebra.

When the base field $K$ is the quadratic imaginary number field $k$  which satisfies the assumption that $p$ splits in $k/\QQ$, we write $p=\wp\wp^*$ with $\wp\neq\wp^*$. Also in this case, we assume that $p$ does not divide the class number $h_k$ of $k$. For an $\oo_k$-ideal $\frak{I}$, let $k(\frak{I})$ be the ray class field of conductor $\frak{I}$. For each $n\geq 0$ we write 
$$\textup{Gal}(k(\wp^{n+1})/k)=\textup{Gal}(k(\wp^{n+1})/k(\wp))\times H,$$
where $H$ is isomorphic to $\textup{Gal}(k(\wp)/k)$ by restriction. We set 
$$k_n=k(\wp^{n+1})^H, \,\,\, k_\infty=\bigcup_{n\geq0}k_n.$$
Then $k_\infty/k$ is a $\ZZ_p$-extension and we write $\Gamma:=\textup{Gal}(k_\infty/k)$ also when there is no danger of confusion. The extension $k_\infty/k$ is the unique $\ZZ_p$-extension which is unramified outside $\wp$. The prime $\wp$ is totally ramified in $k_\infty/k$. Let $\frak{f}_L \subset \oo_k$ denote the conductor of $L$ (which is prime to $\wp$ by our assumptions on $\chi$) and let $\frak{f}$ be a multiple of $\frak{f}_L$ which is prime to $\wp$ and which also satisfies the condition that the map $\oo_k^\times\ra(\oo_k/\ff)^\times$ is injective. Attached to a Grossencharacter $\varphi$ of $k$ of infinity type $(1,0)$ and of conductor $\frak{f}$, there is an elliptic curve $E$ defined over $F=k(\frak{f})$ with the properties that
\begin{itemize}
\item $E$ has complex multiplication by $\oo_k$
\item $F(E_\textup{tor})$ is an abelian extension of $k$,
\end{itemize}
where we write $F(E_\textup{tor})$ for the extension of $K$ which is generated by the coordinates of the torsion-submodule $E_{\textup{tor}}\subset E(\overline{k})$. For such $E$, we have $F(E[\wp^{n+1}])=k(\frak{f}\wp^{n+1})$ for all $n\geq 0$, and using this fact one obtains a canonical identification $\textup{Gal}(F(E[\wp^{\infty}])/F(E[\wp]))\stackrel{\sim}{\ra} \Gamma$ and the following isomorphisms:
\begin{itemize} 
\item[(i)] $\rho_E:\textup{Gal}(F(E[\wp^{\infty}])/F) \stackrel{\sim}{\lra} \textup{Aut}(E[\wp^{\infty}])=\oo_{k_\wp}^\times\stackrel{\sim}{\lra}\ZZ_p^\times$,
\item[(ii)] $\rho_\Gamma:=\rho_E\big{|}_{\Gamma}: \Gamma \stackrel{\sim}{\lra} 1+p\ZZ_p$.
\end{itemize}
The character $\rho_\Gamma$ will play the role of cyclotomic character when our base field $K$ is the quadratic imaginary number field $k$.

For any finitely generated abelian group $M$ endowed with a $G_{K}$ action, $\widehat{M}$ will denote its $p$-adic completion $\textup{Hom}(\textup{Hom}(M,\QQ_p/\ZZ_p),\QQ_p/\ZZ_p)$, and $M^\chi$ will denote the $\chi$-isotypic part of $\widehat{M}\otimes_{\ZZ_p}\oo$. Also, let $\log_p:1+p\ZZ_p\ra \ZZ_p$ denote the $p$-adic logarithm.

For a field $K$ (with fixed separable closure $\overline{K}/K$) and a $\oo[[\textup{Gal}(\overline{K}/K)]]$-module $X$ which is finitely generated over $\oo$, we will denote the $i$-th cohomology (with continuous cochains) of the group $\textup{Gal}(\overline{K}/K)$ with coefficients in $X$ by $H^i(K,X)$.

For every positive integer $n$, we define $\mu_n \subset \overline{\QQ}$ to be the set of $n$th roots of unity.

\section{Height pairings on extended Selmer groups}
\label{sec:hts}
\subsection{Generalities}
\label{subsec:tilde}
In this section we very briefly review Nekov\'{a}\v{r}'s theory of Selmer complexes and his definition of extended Selmer groups. The treatment in this section is far more general than what is needed for the purposes of this paper, and it is much less general than what is covered in~\cite{nek}. For example, we focus on coefficient rings such as the ring of integers $\mathcal{O}$ of a finite extension of $\QQ_p$, or the one variable Iwasawa algebra $\mathcal{O}[[\Gamma]]$; and we restrict our attention to a complex of $\mathcal{O}$-modules $M$ of finite type, endowed with a continuous action of the absolute Galois group $G_K$ of a fixed base field $K$, \emph{concentrated in degree zero}. From~\S\ref{subsec:classical} on, $K$ will be $\QQ$ (except in \S\ref{subsec:caseimagquad} where $K=k$, a quadratic imaginary number field and Remark~\ref{rem:totallyrealDDP} where $K$ is an arbitrary totally real field), and $M$ will be one of $\oo(1)\otimes\chi^{-1}$, $\oo(\chi)$, $\oo(1)$ or $\oo$ (in degree zero) . 

Let $G$ be a profinite group (given the profinite topology) and let $\mathcal{O}$ be as above. Let $M$ be a free $\mathcal{O}$-module of finite type on which $G$ acts continuously. Then $M$ is admissible in the sense of~\cite[\S3.2]{nek} and we can talk about the complex of \emph{continuous} cochains $C^\bullet(G,M)$ as in~\S3.4 of {loc.cit}. Let $K$ be a number field with a fixed algebraic closure $\overline{K}$ and let $S$ denote a finite set of primes of $K$ which contains all primes above $p$, all primes at which the representation $M$ is ramified and all infinite places of $K$, let $S_f$ denote the subset of finite places of $S$. Let $K_S$ the maximal sub-extension of $\overline{K}/K$ which is unramified outside $S$, and let $G_{K,S}$ denote the Galois group $\textup{Gal}(K_S/K)$. For all $w \in S_f$, we write $K_w$ for the completion of $K$ at $w$, and $G_w$ for its absolute Galois group. Whenever it is convenient, we will identify $G_w$ with a decomposition subgroup inside $G_K:=\textup{Gal}(\overline{K}/K)$. We will be interested in the cases $G=G_{K,S}$ or $G=G_w$. 

\subsubsection{Selmer complexes}
\label{subsubsec:selmer}
Classical Selmer groups are defined as elements of the global cohomology group $H^1(G_{K,S},M)$ satisfying certain local conditions; see~\cite[\S2.1]{mr02} for the most general definition. The main idea of~\cite{nek} is to impose local conditions on the level of complexes. We go over basics of Nekov\'{a}\v{r}'s theory, for details see~\cite{nek}.

\begin{define}
\label{def:local condition}
\emph{Local conditions} for $M$ are given by a collection $\Delta(M)=\{\Delta_w(M)\}_{w\in S_f}$, where $\Delta_w(M)$ stands for a morphism of complexes of $\mathcal{O}$-modules 
$$i_w^+(M): U_w^+\lra C^\bullet(G_w,M)$$
 for each $w\in S_f$.
\end{define}

Also set 
$$U_v^-(M)=\textup{Cone}\left( U_v^+(M) \stackrel{-i_v^+}{\lra} C^\bullet(G_v,M) \right)$$
and 
$$U_S^{\pm}(M)=\bigoplus_{w\in S_f} U_w^{\pm}(M); \,\,\,\,i_S^+(M)= (i_w^+(M))_{w\in S_f}.$$
We also define 
$$\textup{res}_{S_f}: C^\bullet(G_{K,S},M) \lra \bigoplus_{w\in S_f}C^\bullet(G_w,M)$$ as the canonical restriction morphism.
\begin{define}
\label{def: selmer complex}
The \emph{Selmer complex} associated with the choice of local conditions $\Delta(M)$ on $M$ is given by the complex 
$$\xymatrix@C=.27in{
\widetilde{C}_f^\bullet(G_{K,S},M,\Delta(M)):= \textup{Cone}(C^\bullet(G_{K,S},M)\bigoplus U_S^+(M)\ar[rr]^(.65){\textup{res}_{S_f}-i_{S}^+(M)}&&\bigoplus_{w\in S_f} C^\bullet(G_w,M))[-1]
}$$
 where $[n]$ denotes a shift by $n$. The corresponding object in the derived category will be denoted by $\widetilde{\mathbf{R}\Gamma}_f(G_{K,S},M,\Delta(M))$ and its cohomology by $\widetilde{H}^i_f(G_{K,S},M,\Delta(M))$ (or simply by $\widetilde{H}^i_f(K,M)$ or by $\widetilde{H}^i_f(M)$ when there is no danger of confusion). The $\oo$-module $\widetilde{H}^1_f(M)$ will be called the \emph{extended Selmer group}.

The object in the derived category corresponding to the complex $C^\bullet(G_{K,S},M)$ will be denoted by ${\mathbf{R}\Gamma}(G_{K,S},M)$.
\end{define}

\subsubsection{Comparison with classical Selmer groups}
\label{subsubsec:comparison}
For each $w \in S_f$, suppose that we are given a submodule 
$$H^1_\FF(K_w,M)\subset H^1(K_w,M).$$
This data which  $\FF$ encodes is called a \emph{Selmer structure} on $M$. Starting with $\FF$, one defines the Selmer group as 
$$H^1_\FF(K,M):=\ker\left\{H^1(G_{K,S},M)\lra \bigoplus_{w\in S_f}\frac{H^1(K_w,M)}{H^1_\FF(K_w,M)}\right\}.$$

On the other hand,  as explained in~\cite[\S6.1.3.1-\S6.1.3.2]{nek}, there is an exact triangle
$$
U_S^-(M)[-1]\lra \widetilde{\mathbf{R}\Gamma}_f(G_{K,S},M,\Delta(M)) \lra {\mathbf{R}\Gamma}(G_{K,S},M)\lra U_S^-(M)
$$
This gives rise to an exact sequence in the level of cohomology:
\begin{prop}[\textup{\cite[\S0.8.0 and \S9.6]{nek}}]\label{prop:compare1} 
For each $i$, the following sequence is exact:
$$\dots\lra H^{i-1}(U_S^-(M))\lra \widetilde{H}^i_f(M) \lra H^i(G_{K,S},M)\lra H^i(U_S^-(M))\lra \dots$$
\end{prop}

This proposition is used to compare Nekov\'{a}\v{r}'s extended Selmer groups to classical Selmer groups. Although this may be achieved in greater generality, we will only state the relevant comparison theorem for \emph{Greenberg's local conditions} (and \emph{Greenberg's Selmer groups}) whose definitions we now recall. For further details, see~\cite{g1, gr, nek}.

Let ${I}_w$ denote the inertia subgroup of $G_w$. Suppose we are given an $\mathcal{O}[[G_w]]$-submodule $M^+_w$ of $M$ for each place $w|p$ of $K$, set $M^-_w=M/M^+_w$. Then Greenberg's local conditions (on the complex level, i.e., in the sense of~\cite[\S6]{nek}) are given by
$$U_w^+=\left\{
\begin{array}{cl}
  C^\bullet(G_w,M_w^+)& \hbox{ if } w|p,    \\\\
  C^\bullet(G_w/I_w,M^{I_w})& \hbox{ if } w\nmid p
\end{array}
\right.$$
with the obvious choice of morphisms 
$$i_w^+(M): U_w^+(M)\lra C^\bullet(G_w,M).$$
As in Definition~\ref{def: selmer complex}, we then obtain a Selmer complex and an extended Selmer group, which we denote by $\widetilde{H}^1_f(M)$. Greenberg's local conditions are the only type of local conditions we will deal with from now on.

We now define the relevant Selmer structure\footnote{For a general $M$, our definition of $\FFc$ (\emph{the canonical Selmer structure}) slightly differs from its original definition in~\cite{mr02}. However, for the specific Galois representation we use starting from \S\ref{subsec:classical} on, they do coincide.} $\FFc$ on $M$.
\begin{define}
\label{def:can-selmer}
The \emph{canonical Selmer structure} $\FFc$ is given by
$$H^1_{\FFc}(K_w,M)=\left\{
\begin{array}{cl}
  \begin{array}{l} \textup{im}\left(H^1(G_w,M_w^+)\ra H^1(K_w,M)\right)=\\ \ker\left(H^1(G_w,M_w)\ra H^1(G_w,M_w^-)\right)  \end{array}& \hbox{ \,\,\,\,if } w|p,    \\\\
 \begin{array}{l}\ker \left(H^1(G_w,M)\ra H^1(I_w,M) \right)= \\ \textup{im}\left(H^1(G_w/I_w,M^{I_w}) \ra H^1(G_w,M) \right)   \end{array}& \hbox{\,\,\,\, if } w\nmid p.
\end{array}
\right.
$$
\end{define}

Hence, we obtain the following Selmer group (which is called the \emph{strict Selmer group} in~\cite[\S9.6.1]{nek} and denoted by $S_M^{\textup{str}}(K)$): 
\be
\label{eqn:selmer group}
H^1_{\FFc}(K,M)=\ker\left(H^1(G_{K,S},M)\lra \bigoplus_{w|p}H^1(G_w,M_w^-)\oplus \bigoplus_{w\nmid p} H^1(I_w,M)   \right).
\ee

Proposition~\ref{prop:compare1} now shows that:
\begin{prop}
\label{prop:compare}
The following sequence is exact:
$$M^{G_K} \lra \bigoplus_{w|p}(M_w^-)^{G_w} \lra \widetilde{H}^1_f(M)\lra H^1_{\FFc}(K,M)\lra 0.$$
\end{prop}
See~\cite[Lemma~9.6.3]{nek} for a proof.
\begin{rem}
Note that if $(M_w^-)^{G_w}=0$ for all $w|p$, then the extended Selmer group $\widetilde{H}^1_f(M)$ coincides with the canonical Selmer group $H^1_{\FFc}(K,M)$. However, if some $(M_w^-)^{G_w}\neq 0$ then $\widetilde{H}^1_f(M)$ is strictly larger than $H^1_{\FFc}(K,M)$ (under the assumption that $M^{G_K}$=0, say). This is the main feature of Nekov\'{a}\v{r}'s Selmer complexes: They reflect the existence of exceptional zeros, unlike classical Selmer groups.
\end{rem}
\subsubsection{Height pairings}
\label{subsubsec:heights}
We now recall Nekov\'{a}\v{r}'s definition of height pairings on his extended Selmer groups. All the references in this section are to~\cite[\S11]{nek} unless otherwise stated. 

Let $M^*=\textup{Hom}(M,\mathcal{O})(1)$ (in Nekov\'{a}\v{r}'s language this is $\mathcal{D}(M)(1)$, the Grothendieck dual of $M$). Let  $\Gamma$ be the Galois group $\textup{Gal}(\QQ_{\infty}/\QQ)$ (resp., the Galois group $\textup{Gal}(k_\infty/k)$) and $\rho$ be the cyclotomic character $\rho_{\textup{cyc}}$ (resp., the character $\rho_{\Gamma}$) when the base field $K$ is $\QQ$ (also more generally, when $K$ is a totally real number field) (resp., when $K$ is the quadratic imaginary number field $k$). The height pairing 
$$\xymatrix{
\langle\,,\,\rangle_{\textup{Nek}}:\,\, \widetilde{H}^1_f(M) \otimes_{\mathcal{O}}\widetilde{H}^1_f(M^*)\ar[r]& \mathcal{O}\otimes_{\ZZ_p}\Gamma\ar[rr]^(.59){\textup{id}\,\otimes\,\log_p\rho}&&\oo
}$$
 is defined in two steps:
\begin{itemize}
\item[(i)] Apply the \emph{Bockstein} morphism 
$$\xymatrix{\beta:\widetilde{\mathbf{R}\Gamma}_f(M) \ar[r]& \widetilde{\mathbf{R}\Gamma}_f(M)[1]\otimes_{\ZZ_p}\Gamma\ar[rr]^(.56){\textup{id}\,\otimes\,\log_p\rho}&&\widetilde{\mathbf{R}\Gamma}_f(M)[1]}$$
  See~\cite[\S11.1.3]{nek} for the original definition of $\beta$. Let $\beta^1$ denote the map induced on the level of cohomology: 
$$\beta^1:\,\, \widetilde{H}_f^1(M) \lra \widetilde{H}^2_f(M).$$
\item[(ii)] Use the \emph{global duality} pairing 
$$\langle\,,\,\rangle_{\textup{PT}}:\,\, \widetilde{H}^2_f(M)\otimes_{\mathcal{O}}\widetilde{H}^1_f(M^*) \lra \mathcal{O}$$
 on the image of $\beta^1$ inside of $\widetilde{H}^2_f(M)$. Here the subscript PT stands for Poitou-Tate, and the global pairing comes from summing up the invariants of the local cup product pairing, see~\cite[\S6.3]{nek} for more details.
\end{itemize}

Just as for other height pairings, universal norms are in the kernel of Nekov\'a\v{r}'s height pairing:
\begin{prop}[{\cite[Proposition 11.5.7 and \S11.5.8]{nek}}]
\label{prop:kernel}
For $X=M, M^*$, the universal norms 
$$\textup{im}\left(\widetilde{H}^1_f(G_{K,S},X\otimes_{\mathcal{O}}\mathcal{O}[[\Gamma]],\Delta(M)\otimes\mathcal{O}[[\Gamma]])\lra \widetilde{H}^1_f(X)\right)$$
 are in the kernel of the height pairing $\langle\,,\,\rangle_{\textup{Nek}}$.
\end{prop}
 Here $\Delta(M)\otimes\mathcal{O}[[\Gamma]]$ stands for an appropriate propagation of the local conditions $\Delta(M)$ on $M$ to $M\otimes_{\mathcal{O}}\mathcal{O}[[\Gamma]]$, see \cite[\S8]{nek} (particularly \S8.6) for details. 
\subsection{The classical case: $T=\oo(1)\otimes\chi^{-1}$}
\label{subsec:classical}
In this section we explicitly calculate both the classical Selmer groups and the extended Selmer groups associated with the representations $T=\oo(1)\otimes\chi^{-1}$ and $T^*=\oo(\chi)$, viewed as a representation of $G_{K}$. We keep the notation of~\S\ref{subsec:tilde}. Let $S=\{\frak{q}: \frak{q}\mid p\frak{f}_\chi\infty\}$ be a set of places of $K$. We set $T^+=T$, $(T^*)^+=0$ (hence $T^-=0$, $(T^*)^-=T^*$).

\begin{lemma}
\label{lem:comparetilde}
\begin{itemize}
\item[(i)] $\widetilde{H}^1_f(K,T) \stackrel{\sim}{\lra} H^1_{\FFc}(K,T),$
\item[(ii)] The sequence 
$$0\lra \bigoplus_{\wp|p}H^0(K_\wp,\oo(\chi)) \lra \widetilde{H}^1_f(K,T^*) \lra H^1_{\FFc}(K,T^*)\lra 0$$ is exact.
\end{itemize}
\end{lemma}
\begin{proof}
Immediate from Proposition \ref{prop:compare}. 
\end{proof}

\begin{rem}
\label{rem:changingS}
For our particular Galois representation $T$, the Selmer group $H^1_{\FFc}(K,T)$ as defined above agrees with what \cite{mr02} calls $H^1_{\FFc}(K,T)$. Indeed, in the language of \cite{mr02}, $H^1_{\FFc}(\QQ,T)$ is defined as 
$$H^1_{\FFc}(K,T)=\ker\left(H^1(G_{K,S},T) \lra \bigoplus_{\frak{q}\in S, \frak{q} \nmid p} \frac{H^1(K_\frak{q},T)}{H^1_{{f}}(K_\frak{q},T)} \right)$$ 
where $\ff=\ff_\chi$ denotes the conductor of $\chi$, and $H^1_{{f}}(K_\frak{q},T)\subset H^1(K_\frak{q},T) $ is as in \cite[Definition I.3.4]{r00}. Let 
$$H^1_{\textup{ur}}(K_\frak{q},T)=\ker(H^1(K_\frak{q},T)\lra H^1(I_\frak{q},T)).$$ 
It follows from~\cite[Lemma~I.3.5(iii)]{r00} that 
$$H^1_{{f}}(K_\frak{q},T)=H^1_{\textup{ur}}(K_\frak{q},T)$$
 for every $\frak{q}\nmid p$ (including primes $\frak{q} | \ff_\chi$), hence it follows that the canonical Selmer group of~\cite{mr02} is given by 
 $$H^1_{\FFc}(K,T)=\ker\left(H^1(G_{K,S},T) \lra \bigoplus_{\frak{q} \in S, \frak{q}\nmid p} H^1(I_\frak{q},T) \right).$$
This shows that our definition of the canonical Selmer group given by (\ref{eqn:selmer group}) agrees with the definition of~\cite{mr02}.
\end{rem}

\begin{prop}
\label{prop:classical explicit}
Let  $\mathcal{O}_L$ denote the ring of integers of $L$, $\mathcal{O}_L\left[{1}/{p}\right]$ its $p$-integers, $\mathcal{O}_L^\times$ its unit group and $\mathcal{O}_L\left[{1}/{p}\right]^\times$ its $p$-units.

\begin{itemize}
\item[(i)] $H^1_{\FFc}(K,T)=\left(\mathcal{O}_L\left[{1}/{p}\right]^\times\right)^{\chi}$, 
\item[(ii)] $H^1_{\FFc}(K,T^*)=0.$
\end{itemize}
\end{prop}
\begin{proof}
The first part follows from Remark~\ref{rem:changingS} and \cite{mr02}~Equation~(25). For the second part, observe that $H^1_{\FFc}(\QQ,T^*)$ is contained in the submodule of unramified homomorphisms inside 
$$H^1(K,T^*)=\textup{Hom}(G_L,\oo)^{\chi^{-1}}, $$ 
where the equality is obtained from the inflation-restriction sequence. In other words,
 $$H^1_{\FFc}(K,T^*) \subset \textup{Hom}(\textup{Gal}(H_L/L),\oo)^{\chi^{-1}}$$ 
 where $H_L$ denotes the Hilbert class field of $L$. But  since $\textup{Gal}(H_L/L)$ is finite,  
 we have $\textup{Hom}(\textup{Gal}(H_L/L),\oo)=0$, so $H^1_{\FFc}(K,T^*)=0$ as well.
\end{proof}
\begin{cor}
\label{cor:tilde explicit} Keep the notation above.
\begin{itemize}
\item[(i)] $\widetilde{H}^1_{f}(K,T)=\left(\mathcal{O}_L\left[{1}/{p}\right]^\times\right)^{\chi}$, 
\item[(ii)] $\bigoplus_{\wp|p}H^0(K_\wp,\oo(\chi)) \stackrel{\sim}{\lra}\widetilde{H}^1_{f}(K,T^*).$
\end{itemize}
\end{cor}
We suppose until the end of this paper that 
\begin{itemize}
\item[$\mathbf{(H)}$] $\chi(\wp)=1$ for a prime $\wp \subset K$ lying above $p$, and that $\chi(\wp^\prime)\neq 1$  for any other $\wp^\prime \subset K$ above $p$.
\end{itemize}
It follows from Corollary~\ref{cor:tilde explicit} that $\widetilde{H}^1_{f}(\QQ,T^*)$ is a free $\oo$-module of rank one. Furthermore, it follows from the proof of \cite[Proposition III.2.6(ii)]{r00} that we have 
$$\left(\mathcal{O}_L\left[{1}/{p}\right]^\times\right)^{\chi}=\left(\mathcal{O}_L\left[{1}/{\wp}\right]^\times\right)^{\chi}$$ 
since we assume $\mathbf{(H)}$.

When $K=\QQ$ and $\chi$ is an even character, it follows from~\cite[Theorem 5.2.15]{mr02} that the core Selmer rank of the canonical Selmer structure (in the sense of Definition 4.1.11 of {loc.cit.}, see also Corollary 5.2.6 of {loc.cit.}) is $2$ (since we assumed $\chi$ is even and $\chi(p)=1$); hence $H^1_{\FFc}(\QQ,T)=\widetilde{H}^1_{f}(\QQ,T)$ is a free $\oo$-module of rank $2$. We will later describe an explicit $\frak{F}$-basis for  $\widetilde{H}^1_{f}(\QQ,T)\otimes\frak{F}$.

When $K$ is totally real and and $\chi$ is totally odd, then $\left(\mathcal{O}_L\left[{1}/{p}\right]^\times\right)^{\chi}=\left(\mathcal{O}_L\left[{1}/{\wp}\right]^\times\right)^{\chi}$ (resp., $\mathcal{O}_L^{\times,\chi}$) is a free $\oo$-module of rank one (resp., of rank zero) and hence $\widetilde{H}^1_{f}(K,T)$ is also free of rank one.

Let $\beta_\chi^1: \widetilde{H}^1_{f}(\QQ,T)\ra \widetilde{H}^2_{f}(\QQ,T)$ denote the Bockstein morphism, as in~\S\ref{subsubsec:heights} above. 

\begin{prop}
\label{prop:classicalpairing}
For any $x \in \widetilde{H}^1_f(K,T)$ and $y \in \widetilde{H}^1_f(K,T^*)$, 
$$\langle x,y \rangle_{\textup{Nek}}=\langle \beta_\chi^1(x),y\rangle_{\textup{PT}}.$$
\end{prop}
\begin{proof}
This is just a restatement of the definition of Nekov\'{a}\v{r}'s height pairing we gave in \S~\ref{subsubsec:heights}. 
\end{proof}

\section{Cyclotomic units}
\label{sec:cyclo}
Throughout~\S\ref{sec:cyclo}, our base field $K$ is $\QQ$ and $\chi$ is an even, non-trivial Dirichlet character whose order is prime to $p$ and which has the property that $\chi(p)=1$. Let $L$ be the field cut by $\chi$ and write $\Delta:=\textup{Gal}(L/\QQ)$. We set $e_\chi:=\sum_{\delta \in \Delta}\chi^{-1}(\delta)\delta \in \oo[\Delta]$. In this section, we define three different types of special elements which will be crucial in what follows: \emph{Tame cyclotomic units}, \emph{wild cyclotomic units} and Solomon's \emph{wild cyclotomic p-units} defined as in~\cite{sol-wild}. 

Fix a collection $\{\zeta_m:m\geq1\}$ such that $\zeta_m$ is a primitive $m$-th root of unity and $\zeta_{mn}^n=\zeta_m$ for every $m$ and $n$. Let $f=f_\chi$ denote the conductor of $\chi$, and recall the Kummer map which induces a canonical map 
$$F^\times \lra H^1(F,\ZZ_p(1))$$
 for every finite abelian extension $F$ of $\QQ$. 

\begin{define}
\label{def:tame cyclotomic unit} 
For every positive integer $n$ prime to $p$, define 
$${c}_n=\textup{N}_{\QQ(\mu_{nf})/L(\mu_n)}(\zeta_{nf}-1) \in L(\mu_n)^{\times}$$ and,
$${c}_n^{\chi}=e_\chi\textup{N}_{\QQ(\mu_{nf})/L(\mu_n)}(\zeta_{nf}-1) \in L(\mu_n)^{\times,\chi}=H^1(\QQ(\mu_n),T).$$ 
The collection $\mathbf{c}=\{c_n^{\chi}: {(n,p)=1}\}$ is called the collection of \emph{tame $\chi$-cyclotomic units}. The element $c_1^\chi$ is called the \emph{tame $\chi$-cyclotomic unit of} $L$, or simply the \emph{tame cyclotomic unit} once $\chi$ (thus also $L$) is fixed.
\end{define}


For every finite abelian extension $F$ of $\QQ$ of conductor $m$, define $\xi_F=\mathbf{N}_{\QQ(\mu_{mp})/F}(\zeta_{mp}-1).$ 
Here and elsewhere in this paper, the symbol $\mathbf{N}$ stands for the norm map.

Let $\QQ_\infty$ be the cyclotomic $\ZZ_p$-extension of $\QQ$, and $\QQ_n$ be its unique sub-extension of degree $p^n$ over $\QQ$. We set $L_n:=L\QQ_n$. Note that the collection $\{\xi_F\}$ satisfies the Euler system distribution relation, in particular the collection $\{\xi_{L_n}:n \geq1\}$ is norm-coherent. 

\begin{define}
\label{def:wild cyclotomic unit}
The collection $$\xi=\xi_\infty^{\chi}:=\{e_\chi\xi_{L_n}:n\geq1\} \in \varprojlim_n H^1(\QQ_n,T)$$ is called the \emph{wild $\chi$-cyclotomic units}. When $\chi$ is understood, this collection will be called the collection of \emph{wild cyclotomic units}.  
\end{define}

\subsection{Cyclotomic units and `exceptional zeros'}
From our assumption that $\chi(p)=1$, it follows  that $p$ splits completely in $L$. 
\label{subsec:cyclo}
\begin{lemma}
\label{lem:vanishing}
Under the running assumptions $\xi_L=1$.
\end{lemma}

\begin{proof}
This is \cite[ Lemma 2.2]{sol-wild}; see also \cite[Remark 6.1.10]{mr02}.
\end{proof}

Let $\Gamma=\textup{Gal}(\QQ_\infty/\QQ)$ and $\LL=\oo[[\Gamma]]$. Let $\log_p:\ZZ_p^\times \ra \ZZ_p$ be the $p$-adic logarithm, and let $\rho_{\textup{cyc}}:\Gamma\ra1+p\ZZ_p$ be the cyclotomic character. Fix a topological generator $\gamma$ of $\Gamma$. The short exact sequence 
$$0\lra T\otimes\LL\stackrel{\gamma-1}{\lra}T\otimes\LL\lra T\lra0$$
 induces a long exact sequence of cohomology (where we have the zero on the left thanks to our assumption that $\chi$ is non-trivial)
\be
\label{eqn:long-exact}
0=H^0(\QQ,T)\lra H^1(\QQ,T\otimes\LL) \stackrel{\gamma-1}{\lra} H^1(\QQ,T\otimes\LL)\stackrel{\mathbb{N}}{\lra}H^1(\QQ,T).
\ee

By~\cite[Proposition II.1.1]{colmez-reciprocity}, we may identify $H^1(\QQ,T\otimes\LL)$ with $\varprojlim_n H^1(\QQ_n,T)$, and thus view the wild cyclotomic unit $\xi$ as an element of $H^1(\QQ,T\otimes\LL)$. The image of $\xi$ under the map $\mathbb{N}$ of (\ref{eqn:long-exact})  is $\xi_L^{\chi}=1$, hence the exact sequence (\ref{eqn:long-exact}) shows: 

\begin{prop}
\label{prop:division}
There exists a unique $\{z_n^{\chi}\}=z_\infty^{\chi} \in H^1(\QQ,T\otimes\LL)=\varprojlim_n H^1(\QQ_n,T)$
such that 
$$\frac{\gamma-1}{\log_p\rho_{\textup{cyc}}(\gamma)}\times z_\infty^{\chi}=\xi.$$
\end{prop} 

\begin{rem}
\label{rem:comparison with solomon}
Just as we did above, one could have obtained an element $z_\infty \in \varprojlim_n H^1(L_n,\ZZ_p(1))$ such that 
$\frac{\gamma-1}{\log_p\rho_{\textup{cyc}}(\gamma)}\times z_\infty=\xi_\infty:=\{\xi_n\}.$ Then, 
 $\chi$-part of this element would be our $z_\infty^\chi$ and $\xi_\infty^\chi=\xi$, respectively. Although we only need to analyze the $\chi$-parts $z_\infty^\chi$ and $\xi=\xi_\infty^\chi$ of these elements for our purposes, it may be worthwhile to keep this in mind for a comparison with the treatment of~\cite{sol-wild} and~\cite[\S 9.3]{BG}.
\end{rem}

\subsection{Wild cyclotomic $p$-units}
\label{subsec:wild}
In this section we quickly review Solomon's~\cite{sol-wild}
construction of \emph{cyclotomic p-units} and relate these $p$-units to
$z_\infty^\chi$ defined above.

Solomon's construction\footnote{The attentive reader will notice that Solomon's construction is carried out without taking $\chi$-parts. However his arguments apply on the $\chi$-parts verbatim. In fact, it is easy to see that the $p$-unit $\kappa^\chi$ constructed below is simply the $\chi$-part of the $p$-unit $\kappa$ which Solomon constructs in~\cite[\S2]{sol-wild}.} starts with the observation that
there exists (thanks to Hilbert 90) a unique $\beta_n^\chi \in L_n^{\times,\chi}/L^{\times,\chi}$ such that
$$\frac{\gamma-1}{\log_p\rho_{\textup{cyc}}(\gamma)}\times\beta_n^\chi=\xi_n^\chi.$$
 Thus, from  our definition of $z_\infty^\chi=\{z_n^\chi\}$ it follows that 
 $$\beta_n^\chi=z_n^\chi \hbox{ inside } L_n^{\times,\chi}/L^{\times,\chi}.$$ 
 Applying $\mathbf{N}_{L_n/L}$ on both sides of this equality we see that 
\be
\label{def:sol-cyclo}
\kappa_n^\chi:=\mathbf{N}_{L_n/L}\beta_n^\chi\equiv \mathbf{N}_{L_n/L}z_n^\chi=z_0^\chi \mod p^n.
\ee
Solomon proves (and (\ref{def:sol-cyclo}) above shows as well) that
$$\kappa_{n^\prime}^\chi \equiv \kappa_{n}^\chi\mod p^n\,,\,\hbox{ for } n^\prime\geq n,$$ and he defines 
$$\kappa^\chi:=\varprojlim \kappa_n^\chi \in L^{\times,\chi}.$$ 
This is what he calls the \emph{cyclotomic p-unit}. By (\ref{def:sol-cyclo}), we clearly have $\kappa^\chi=z_0^\chi$.

\begin{define}
\label{def:wild cyclo}
The element $z_0^\chi$ is called the \emph{cyclotomic p-unit} and the collection 
$$z_\infty^\chi \in \varprojlim_n H^1(\QQ_{n,p},T)=\varprojlim_n L_n^{\times,\chi}$$
 is called the collection of \emph{wild cyclotomic p-units}.
\end{define}
\begin{rem}
\label{rem:generation of p-units}
By~\cite[Remark 4.4]{sol94} that $\{c_1^\chi,z_0^\chi\}$ is an ordered $\frak{F}$-basis for $\widetilde{H}^1_f(\QQ,T)\otimes\frak{F}$. 
\end{rem}

\subsection{Local Tate duality}
\label{subsec:tate}
In this section we give a review of well-known results from local
duality which we will need later in \S\ref{sec:calculations}. For each
$n\geq 0$, we have the local Tate pairing 
$$ H^1(\QQ_{n,p},T)\times H^1(\QQ_{n,p},T^*)\lra \oo,$$
 induced from cup-product pairing
composed with the invariant isomorphism, see~\cite[\S5.1-\S5.2]{nek} for
more details. This induces a map
$$H^1(\QQ_{n,p},T)\stackrel{\tau_n}{\lra}\textup{Hom}(H^1(\QQ_{n,p},T^*),\oo)$$
thus, in the limit a map (using~\cite[Proposition~II.1.1]{colmez-reciprocity} once again)
$$H^1(\QQ_p,T\otimes\LL)\stackrel{\tau_\infty}{\lra}\textup{Hom}(\varinjlim_n
H^1(\QQ_{n,p},T^*),\oo).$$
\begin{define}
\label{def:maps}
\begin{enumerate}
\item Let $\mathcal{L}_\xi$ be the image of $\xi$ under the compositum
$$
\xymatrix{H^1(\QQ,T\otimes\LL) \ar[r]^{\textup{loc}_p} & H^1(\QQ_p,T\otimes\LL)\ar[r]^(.385){\tau_\infty}& \textup{Hom}(\varinjlim_n H^1(\QQ_{n,p},T^*),\oo).}$$
\item Let $\mathcal{L}^\prime_{\xi}$ be the image of $z_\infty^\chi$ under the compositum
$$\xymatrix{H^1(\QQ,T\otimes\LL) \ar[rr]^(.4){\tau_\infty\,\circ\,\textup{loc}_p}&&\textup{Hom}(\varinjlim_n H^1(\QQ_{n,p},T^*),\oo)\ar[r]
&\textup{Hom}(H^1(\QQ_p,T^*),\oo).
}$$ 
\end{enumerate}
\end{define}

\begin{rem}
\label{rem:local calculation}
For $n\geq n^\prime$ we have a commutative diagram
$$\xymatrix{H^1(\QQ_{n,p},T) \ar[r]\ar[d]_{\mathbf{N}}&
\textup{Hom}\left(H^1(\QQ_{n,p},T^*),\oo\right) \ar[d]^{\textup{res}^*}\\
H^1(\QQ_{n^\prime,p},T) \ar[r] & \textup{Hom}\left(
H^1(\QQ_{n^\prime,p},T^*),\oo\right) }$$
 where $\textup{res}^*$ is
induced from the restriction map 
$$\textup{res}:
H^1(\QQ_{n^\prime,p},T^*) \lra H^1(\QQ_{n,p},T^*).$$ 
We therefore have a commutative diagram 
$$\xymatrix@C=.16in{z_\infty^\chi \ar@{|->}[d]&\in&
H^1(\QQ_p,T\otimes\LL) \ar[r]\ar[d]& \textup{Hom}(\varinjlim_n
H^1(\QQ_{n,p},T^*),\oo) \ar[d]\\ 
z_0^\chi&\in& H^1(\QQ_p.T)\ar[r]&
\textup{Hom}(H^1(\QQ_p,T^*),\oo) } $$ 
Thus $\mathcal{L}_\xi^\prime$ is simply the image of $z_0^\chi$ under the map 
$$\tau_0: H^1(\QQ_p,T) \lra \textup{Hom}(H^1(\QQ_p,T^*),\oo).$$ 
\end{rem}

\section{Computation of the height pairing}
\label{sec:calculations}
Throughout~\S\ref{sec:calculations}, our base field $K$ is $\QQ$ and $\chi$ is an even, non-trivial Dirichlet character whose order is prime to $p$, and which has the property that $\chi(p)=1$. In this section we calculate the height pairing on the cyclotomic unit $c_1^\chi$. Note that, in view of Remark~\ref{rem:generation of p-units}, Proposition~\ref{prop:kernel} and the fact that $z_0^\chi \in \widetilde{H}^1_f(\QQ,T)$ is a universal norm (by its definition), this gives the only non-trivial output of the machinery we described in~\S\ref{sec:hts} we could hope for. 

For  $\psi=\chi^{\pm1}$, we write as usual $\oo(\psi)$ for the free $\oo$-module of rank one, on which $G_\QQ$ acts via $\psi$. Define $e_\psi:= \sum_{\delta\in\Delta}\psi^{-1}(\delta)\delta$ as the idempotent of $\oo[\Delta]$ associated to $\psi$. We identify the module $\oo(\psi)$ with $\left(\oplus_{v|p} \oo\cdot v\right)^{\psi}$ (therefore we regard $\mathfrak{g}_\psi:=e_\psi v_0$ as a generator of $\oo(\psi)$, where we recall that $v_0$ is the place of $L$ we fixed in \S\ref{subsec:notation} via choosing an embedding $\iota_p: \overline{\QQ}\hookrightarrow \overline{\QQ}_p$) and we define 
$$\frak{pr}_\psi: \left(\oplus_{v|p} \oo\cdot v\right)^{\psi} \stackrel{\sim}{\lra} \oo$$
 by setting $\frak{pr}_\psi: \frak{g}_\psi \mapsto 1$. In other words, $\frak{pr}_\psi$ is the map induced from projection onto the $v_0$-coordinate. For each  place $v$ of $L$ lying above $p$, write $\sigma_v: L \hookrightarrow L_v=\QQ_p$ for the induced embedding. 

 Let $\rho_\chi$ denote the compositum 
$$
\xymatrix@C=.27in{
 \widetilde{H}^1_f(\QQ,\oo(1)\otimes\chi^{-1})\ar[r]^{\beta_\chi^{1}}\ar@{-->}[rrd]_{\rho_\chi}& \widetilde{H}^2_f(\QQ,\oo(1)\otimes\chi^{-1})\ar
 [r]^{\iota}& H^2(\QQ,\oo(1)\otimes{\chi^{-1}})\ar[d]\\
&& H^2(\QQ_p,\oo(1)\otimes{\chi^{-1}}) 
}$$
and $\beta_\chi$ the compositum
$$
\xymatrix@C=.24in{
 \widetilde{H}^1_f(\QQ,\oo(1)\otimes\chi^{-1})\ar[r]^{\rho_\chi}\ar@{-->}[rrdd]_{\beta_\chi} &H^2(\QQ_p,\oo(1)\otimes{\chi^{-1}})\ar[r]^(.48){\cong}_(,48){\frak{h}}&\ \left(\bigoplus_{v|p}H^2(L_v,\oo(1))\right)^{\chi}\ar[d]^{\sum_{v}\textup{inv}_v}\\
 &&\left(\bigoplus_{v|p} \oo\cdot v \right)^{\chi} \ar[d]_{\cong}^{\frak{pr}_\chi}\\
 &&\oo
}$$
where the map $\beta_\chi^{1}$ in the first diagram is the Bockstein morphism applied on the first cohomology;  $\iota$ comes from Proposition~\ref{prop:compare1}; the isomorphism $\frak{h}$ in the second diagram from the Hochschild-Serre spectral sequence.  
Let $\log_p: \widehat{\QQ_p^\times} \ra \ZZ_p$ be the $p$-adic logarithm extended to the $p$-adic completion $\widehat{\QQ_p^\times}$ of $\QQ_p^\times$ by setting $\log_p(p)=0$. We extend $\log_p$ by linearity to define an $\oo$-module homomorphism
$$\log_p: \oo\otimes_{\ZZ_p}\widehat{\QQ_p^\times} \lra \oo.$$
 \begin{prop}
 \label{prop:bockstein calculation}
$\beta_\chi(c_1^{\chi})=\log_p(\iota_p(c_1^\chi))=v_0(z_0^\chi) \in \oo.$
 \end{prop}
\begin{proof}
The second equality is the main calculation of~\cite{sol-wild}, hence it suffices to check the first claimed equality. This assertion is  essentially~\cite[Proposition 9.3(ii)]{BG}. In fact, the statement of {loc.cit.} is that 
$\beta_\chi(c_1^\chi)=\frak{pr}_\chi\left(e_\chi\sum_{v|p}\log_p(\sigma_v(c_1))\cdot v\right),$
 where the equality takes place in $\oo$. Furthermore, we have the following brute-force calculation:
\begin{align*}
\oo(\chi)\ni e_\chi\sum_{v|p}\log_p(\sigma_v(c_1))\cdot v&= \sum_{\delta\in\Delta}\chi^{-1}(\delta)\delta\sum_{v|p}\log_p(\sigma_v(c_1))\cdot v\\
&= \sum_{\delta\in\Delta}\sum_{v|p}\chi^{-1}(\delta)\log_p(\sigma_v(c_1))\cdot v^{\delta}\\
&=\sum_{\delta\in\Delta}\sum_{\omega|p}\chi^{-1}(\delta)\log_p(\sigma_{\omega^{\delta^{-1}}}(c_1))\cdot \omega\\
&=\sum_{\delta\in\Delta}\sum_{\omega|p}\chi^{-1}(\delta)\log_p(\sigma_{\omega}(c_1^\delta))\cdot \omega\\
&=\sum_{\omega|p}\log_p(\sigma_{\omega}(c_1^\chi))\cdot \omega \in \oo(\chi),
\end{align*}
 where $v^{\delta}$ is the place obtained by the action of $\delta \in \Delta$ on the set of places $\{v: v|p\}$; and we have the final equality by the $\oo$-linearity of $\log_p$, and the forth equality thanks to the following commutative diagram:
 \be\label{diag:commutative}\xymatrix@R=.15in @C=.8in{ 
 L \ar[d]_{\delta}\ar[r]^{\sigma_v}&\QQ_p \ar@{=}[d]\\
 L\ar[r]_{\sigma_{v^{\delta}}}& \QQ_p
 }\ee
  We further have,
  \begin{align*} 
\sum_{\omega|p}\log_p(\sigma_{\omega}(c_1^\chi))\cdot \omega&= \sum_{\delta \in \Delta}\log_p(\sigma_{v_0^\delta}(c_1^\chi))\cdot v_0^\delta\\
&=  \sum_{\delta \in \Delta}\log_p \left(\sigma_{v_0}((c_1^\chi)^{\delta^{-1}})\right)\cdot v_0^\delta\\
&= \sum_{\delta \in \Delta}\log_p\left(\sigma_{v_0}(c_1^\chi)^{\chi^{-1}(\delta)}\right)\cdot v_0^\delta\\
&=\sum_{\delta \in \Delta}{\chi^{-1}(\delta)}\log_p\left(\sigma_{v_0}(c_1^\chi)\right)\cdot v_0^\delta\\
&=\log_p\left(\sigma_{v_0}(c_1^\chi)\right)\cdot e_\chi v_0 \in \oo(\chi),
\end{align*}
where the second equality holds thanks to (\ref{diag:commutative}) and the third because $\left(c_1^\chi\right)^{\delta^{-1}}=\left(c_1^\chi\right)^{\chi^{-1}(\delta)}$. Putting all this together (and noting that $\sigma_{v_0}\big{|}_L=\iota_p\big{|}_L$ by definition), we conclude that
$$\beta_\chi(c_1^\chi)=\xi_\chi\left( \log_p\left(\sigma_{v_0}(c_1^\chi)\right)\cdot e_\chi v_0 \right)=\log_p(\iota_p(c_1^\chi))$$
as desired.
\end{proof}
\begin{rem}
\label{rem:dependenceoflogonv0}
Note that if we replace $v_0$ by another place $v_0^\delta$ of $L$, the value of $\beta_\chi(c_1^\chi)=\log_p(\sigma_{v_0}(c_1^\chi))$ changes by $\chi^{-1}(\delta)$: \,\,\,\,$\log_p(\sigma_{v_0^\delta}(c_1^\chi))=\chi^{-1}(\delta)\log_p(\sigma_{v_0}(c_1^\chi)).$
\end{rem}
We are now ready to complete the computation of Nekov\'a\v{r}'s height pairing $\langle c_1^\chi, \alpha\rangle_{\textup{Nek}}$ for $\alpha \in \widetilde{H}^1_f(\QQ,T^*)$ and $c_1^\chi$ as above. We have the following identifications:
\begin{equation}
\label{eqn:tilde identification}
\widetilde{H}^1_f(\QQ,T^*) \stackrel{\sim}{\lra}H^0(\QQ_p,\oo(\chi))\stackrel{\sim}{\lra}\left(\bigoplus_{v|p}\oo\cdot v\right)^{\chi^{-1}}\stackrel{\frak{pr}_{\chi^{-1}}}{\lra} \oo.
\end{equation} 
Let $\alpha(v_0)$ denote the image of $\alpha$ under the compositum of the maps (\ref{eqn:tilde identification}). 
\begin{rem}
\label{rem:dependenceonv0}
Note that since $\frak{pr}_{\chi^{-1}}$ depends on the choice of $v_0$, so does $\alpha(v_0)\in \oo$. Write $\frak{pr}_{\chi^{-1}}=\frak{pr}_{\chi^{-1}}(v_0)$ only in this remark to remind us the dependence on $v_0$. One then has $\frak{pr}_{\chi^{-1}}(v_0^\delta)=\chi(\delta)\frak{pr}_{\chi^{-1}}(v_0)$ and in turn $\alpha(v_0^\delta)=\chi(\delta)\alpha(v_0)$.
\end{rem}
\begin{lemma}
\label{lem:local pairing-1}
Suppose $x\in H^0(\QQ_p,\oo)=\oo$ and $y \in H^2(\QQ_p,\oo(1))$. Then
\begin{itemize}
\item[(i)] $x\cup y =x\cdot y \in H^2(\QQ_p,\oo(1))$,
\item[(ii)] $\langle x,y\rangle_{\textup{Tate}}= x\cdot \textup{inv}_p(y) \in \oo$, where $\langle\,,\,\rangle_{\textup{Tate}}$ is the local Tate pairing.
\end{itemize}
\end{lemma}

\begin{proof}
Clear.
\end{proof}
Lemma~\ref{lem:local pairing-1} may be used to check the following:
\begin{lemma}
\label{lem:local pairing-2}
The following diagram commutes:
$$\xymatrix@C=0.01in{
H^0(\QQ_p,\oo(\chi))\ar[d]_(.4){\cong}&\otimes & H^2(\QQ_p,\oo(1)\otimes\chi^{-1})\ar[d]^(.4){\cong}\ar[rrrrr]^(.73){\langle\,,\,\rangle_{\textup{Tate}}}&&&&&\oo\ar@{=}[ddd]\\ 
\left(\bigoplus_{v|p}H^0(L_v,\oo)\right)^{\chi^{-1}}\ar@{=}[d]& & \left(\bigoplus_{v|p}H^2(L_v,\oo(1))\right)^{\chi}\ar[d]_{\sum_{v|p}\textup{inv}_v}^\cong&&&&&\\
\left(\bigoplus_{v|p}\oo\cdot v\right)^{\chi^{-1}}\ar[d]_(.6){\frak{pr}_{\chi{-1}}}&&\left(\bigoplus_{v|p}\oo\cdot v\right)^{\chi}\ar[d]_(.6){\frak{pr}_\chi}\\
\oo&\otimes&\oo\ar[rrrrr]_{(\,,\,)}&&&&&\oo
}$$
\end{lemma}
Here, $(a,b):=ab \in \oo$ for $a,b \in \oo$, and the vertical isomorphisms between first two rows come from the Hochschild-Serre spectral sequence.

The following Proposition is key to our main results. 

\begin{prop}
\label{prop:exp-pairing}
For an arbitrary $\alpha \in \tilde{H}^1_f(\QQ,T^*)$, we have $\langle c_1^\chi,\alpha\rangle_{\textup{Nek}}=v_0(z_0^\chi)\cdot\alpha(v_0).$
\end{prop}
\begin{rem}
\label{rem:indpendentofv0}
Both $v_0(z_0^\chi)$ and $\alpha(v_0)$ depend on the choice of $v_0$, yet 
$v_0(z_0^\chi)\cdot\alpha(v_0)$ is independent of $v_0$ thanks to Remarks~\ref{rem:dependenceoflogonv0} and \ref{rem:dependenceonv0}.
\end{rem}
\begin{proof}
By Proposition~\ref{prop:classicalpairing} 
$$\langle c_1^\chi,\alpha\rangle_{\textup{Nek}}=\langle\beta_\chi^{1}(c_1^\chi),\alpha \rangle_{\textup{PT}},$$
 where 
 $$\langle\,,\,\rangle_{\textup{PT}}: \widetilde{H}^2_f(\QQ,T)\otimes\widetilde{H}^1_f(\QQ,T^*) \lra \oo$$
   denotes the global pairing from~\cite[\S6.3]{nek}. The definition of this global pairing (along with the fact that $H^2(\QQ_\ell,T)=0$ for every $\ell|f_\chi$) shows that the following diagram commutes:
\begin{equation}
\label{eqn:pairing}
\xymatrix @C=.2in{\widetilde{H}^2_f(\QQ,\oo(1)\otimes\chi^{-1})\ar[d]_{\textup{loc}_p\,\circ\,\iota}&\otimes&\widetilde{H}^1_f(\QQ,\oo(\chi))\ar[r]^(.67){\langle\,,\,\rangle_{\textup{PT}}}& \oo\\
H^2(\QQ_p,\oo(1)\otimes\chi^{-1})&\otimes&H^{0}(\QQ_p,\oo(\chi))\ar[u]\ar[r]^{\cup}&H^2(\QQ_p,\oo(1))\ar[u]^{\textup{inv}_p}
}
\end{equation}
We explain the arrows in (\ref{eqn:pairing}): The arrow on the left is the canonical map  (coming from Proposition~\ref{prop:compare1})
$$\iota: \widetilde{H}^2_f(\QQ,\oo(1)\otimes\chi^{-1})
\lra H^2(\QQ,\oo(1)\otimes\chi^{-1})$$
 followed by the restriction map $\textup{loc}_p$. The extended Selmer group $\widetilde{H}^1_f(\QQ,\oo(\chi))$ may be canonically identified by $H^0(\QQ_p,\oo(\chi))$  (see \S\ref{subsec:classical}), this is how we obtain the vertical arrow in the center. 

The commutative  diagram (\ref{eqn:pairing}) gives $\langle c_1^\chi,\alpha\rangle_{\textup{Nek}}=\langle \rho_\chi(c_1^\chi),\alpha\rangle_{\textup{Tate}},$
where $\rho_\chi$ is defined as in the beginning of \S\ref{sec:calculations}. Furthermore, by Lemma~\ref{lem:local pairing-2} 
$$\langle \rho_\chi(c_1^\chi),\alpha\rangle_{\textup{Tate}}=\left(\beta_\chi(c_1^\chi),\alpha(v_0)\right)=v_0(z_0^\chi)\cdot\alpha(v_0),$$
 where $(a,b):=a\cdot b$ for $a,b \in \oo$ as in Lemma~\ref{lem:local pairing-2}, and the final equality is Proposition~\ref{prop:bockstein calculation}. The proof is now complete.
\end{proof}

\section{Rubin's formula}
\label{sec:main}
Throughout~\S\ref{sec:calculations}, our base field $K$ is $\QQ$ and $\chi$ is an even, non-trivial Dirichlet character whose order is prime to $p$, and which has the property that $\chi(p)=1$. In this section we complete our main computation, using the calculations carried out in \S\ref{sec:calculations}. Starting with $\alpha \in \widetilde{H}^1_f(\QQ,\oo(\chi))$ as above, we first wish to define an element $\phi_\alpha$
\be\label{eqn:defphialpha}
\xymatrix{
\phi_\alpha\in H^1(\QQ_p,\oo(\chi))=\left(\bigoplus_{v|p}H^1(L_v,\oo)\right)^{\chi^{-1}}\ar[r]^(.6){\sim}_(.6){\frak{pr}_{\chi^{-1}}}&H^1(L_{v_0},\oo)= \textup{Hom}(G_{v_0},\oo).
}
\ee
 Here we recall that $G_v=\textup{Gal}(\overline{\QQ}_p/L_v)$ and $\frak{pr}_{\chi^{-1}}$ is the projection onto the $v_0$-coordinate as in \S\ref{sec:calculations}. In the equalities above, we are again using an identification coming from Hochschild-Serre spectral sequence, along with the fact that $H^1(L_v,\oo)=\textup{Hom}(G_v,\oo)$. Note also that $\textup{Hom}(G_v,\oo)$ is the group of continuous homomorphisms and we have 
 $$\textup{Hom}(G_v,\oo)=\textup{Hom}(G_v^{\textup{ab}},\oo)=\textup{Hom}(G_v^{\textup{ab},p},\oo)=\textup{Hom}_{\oo}(\oo\otimes_{\ZZ_p}G_v^{\textup{ab},p},\oo),$$
 where $G_v^{\textup{ab}}$ for the abelianization of the group $G_v$ and $G_v^{\textup{ab},p}$ is its pro-$p$ part.
 
 We write $\phi_\alpha^{v_0} \in  \textup{Hom}(G_{v_0},\oo)$ for the image of $\phi_\alpha$ under the compositum (\ref{eqn:defphialpha}) (which we henceforth  call $\frak{r}_\chi$). Defining $\phi_\alpha^{v_0}$ as the unramified homomorphism given by 
$$\xymatrix@R=.01in @C=.02in{
\phi_\alpha^{v_0}:&G_{v_0} \ar[rrrr]&&&&\oo\\
&\textup{Fr}_{v_0}\ar@{|->}[rrrr]&&&& \alpha({v}_0),
}
$$    
where $\textup{Fr}_{v_0}$ denotes an arithmetic Frobenius at $v_0$,  we also define $\phi_\alpha \in H^1(\QQ_p,\oo(\chi))$ using the identification $\frak{r}_\chi$. Below, we normalize the local reciprocity isomorphism (and the local invariant map) by letting uniformizers correspond to arithmetic Frobenius elements.

Let $\xi=\xi_{\infty}^{\chi} =\{\xi_n^{\chi}\}\in H^1(\QQ,T\otimes\LL)$ be the collection of \emph{wild} cyclotomic units, as in \S\ref{sec:cyclo}. 
Recall the definition of the element $\mathcal{L}_\xi^{\prime} \in H^1(\QQ_p,T)$ from \S\ref{subsec:tate} which we regard as an element of $\textup{Hom}(H^1(\QQ_p,T^*),\oo)$ via local duality. Recall also the \emph{tame} cyclotomic unit $c_1^\chi \in H^1(\QQ,T)$.
\begin{thm}
\label{main-primitive}
$\langle c_1^\chi,\alpha\rangle_{\textup{Nek}}=\mathcal{L}_\xi^{\prime}(\phi_\alpha)$.
\end{thm}

\begin{proof}
Let $z_0^\chi$ be Solomon's cyclotomic $p$-unit as above. It follows from the discussion in~\S\ref{subsec:tate} that 
\be\label{eqn:infofrom32}\mathcal{L}_\xi^{\prime}(\phi_\alpha)=\langle z_0^{\chi}, \phi_\alpha \rangle_{\textup{Tate}}.\ee
  The computation of the right hand side of Theorem~\ref{main-primitive} is thus reduced to local class field theory. 
  
  Let $\frak{r}_{\chi^{-1}}$ denote the following compositum:
\be
\label{eqn:compositum1}\xymatrix@R=.1in{H^1(\QQ_p,\oo(1)\otimes\chi^{-1})=\left(\bigoplus_{v|p}H^1(L_v,\oo(1))\right)^\chi\ar[r]^(.61){\sim}_(.61){\xi_{\chi}}&H^1(L_{v_0},\oo(1))= \widehat{L_{v_0}^{\times}}\otimes_{\ZZ_p}\oo,
  }
  \ee
  where $\xi_\chi$ is the projection onto the $v_0$-coordinate as above, and $\widehat{L_v^{\times}}$ stands for the $p$-adic completion of the multiplicative group $L_v^\times$. We note that $\frak{r}_{\chi^{-1}}(\textup{loc}_p(z_0^\chi))=\iota_p(z_0^\chi)$, with $\iota_p: L \hookrightarrow L_{v_0}$  is as in the introduction and $\textup{loc}_p: H^1(\QQ,T)\ra H^1(\QQ_p,T)$ is the canonical restriction map, as usual. We then have a commutative diagram
  $$\xymatrix@R=.25in@C=.2in{
  H^1(\QQ_p,T)\ar[d]_{\frak{r}_{\chi^{-1}}}&\otimes &H^1(\QQ_p,T^*)\ar[d]_{\frak{r}_\chi}\ar[rrr]^(.62){\langle\,,\,\rangle_{\textup{Tate}}}&&&\oo\ar@{=}[d]\\
  H^1(L_{v_0},\oo(1))&\otimes &H^1(L_{v_0},\oo)\ar[rrr]^(.62){\langle\,,\,\rangle_{\textup{Tate}}}&&&\oo\\
  }$$
  which translates to 
  \be\label{eqn:willbeuseful}\langle z_0^{\chi}, \phi_\alpha \rangle_{\textup{Tate}}=\langle \iota_p(z_0^{\chi}), \phi_\alpha^{v_0} \rangle_{\textup{Tate}}.
  \ee
  Let
  $$\frak{a}_v: H^1(L_v,\ZZ_p(1))=\widehat{L_v^{\times}} \lra G_v^{\textup{ab},p}$$
   denote the local reciprocity map. Let further 
   $$\frak{a}_v^{\textup{(ur)}}:\widehat{L_v^{\times}} \lra \textup{Gal}(L_v^{\textup{ur}}/L_v)$$
    denote the projection of $\frak{a}_v$ to the Galois group of the maximal unramified extension of $L_v$. We also write $\frak{a}_v$ (resp., $\frak{a}_v^{\textup{(ur)}}$) for the induced map $\oo\otimes_{\ZZ_p}\widehat{L_v^{\times}} \ra \oo\otimes_{\ZZ_p}G_v^{\textup{ab},p}$ (resp., for the map $\oo\otimes_{\ZZ_p}\widehat{L_v^{\times}} \ra \oo\otimes_{\ZZ_p}\textup{Gal}(L_v^{\textup{ur}}/L_v)$).

By the very definition of the local Tate pairing, 
$$\langle \iota_p(z_0^{\chi}), \phi_\alpha^{v_0} \rangle_{\textup{Tate}}=\phi_\alpha^{v_0}\left(\frak{a}_{v_0}(\iota_p(z_0^{\chi}))\right)=\phi_\alpha^{v_0}\left(\frak{a}_{v_0}^{\textup{(ur)}}(\iota_p(z_0^{\chi}))\right)$$ 
where we have the second equality because $\phi_\alpha^{v_0}$ is unramified by construction. Write 
$$\iota_p(z_0^{\chi})=\varpi_{v_0}^{v_0(z_0^{\chi})}\cdot u \in \oo\otimes_{\ZZ_p}\widehat{L_{v_0}^{\times}}=\oo\otimes_{\ZZ_p}\left(\varpi_{v_0}^{\ZZ_p}\oplus\widehat{\oo_{L_{v_0}}^\times}\right),$$ 
where $\varpi_{v_0}$ is a uniformizer of $L_{v_0}$ and $u \in \oo\otimes_{\ZZ_p}\widehat{\oo_{L_{v_0}}^\times}$ is a unit at $v_0$. Then 
$\frak{a}_{v_0}^{\textup{(ur)}}(\iota_p(z_0^{\chi}))=\textup{Fr}_{v_0}^{v_0(z_0^{\chi})}$ since $\frak{a}_{v_0}(u) \in \mathcal{I}_v\subset G_v$, the inertia subgroup at $v$. Thus 
$$\langle \iota_p(z_0^{\chi}), \phi_\alpha^{v_0} \rangle_{\textup{Tate}}=\phi_\alpha^{v_0}\left(\textup{Fr}_{v_0}^{v_0(z_0^{\chi})}\right)=v_0(z_0^\chi)\cdot\phi_\alpha^{v_0}\left(\textup{Fr}_{v_0}\right)$$ and this equals, by the definition of $\phi_\alpha^{v_0}$, to  $v_0(z_0^\chi)\cdot\alpha(v_0)$, which equals, by Proposition~\ref{prop:exp-pairing} to $\langle z_0^\chi,\alpha\rangle_{\textup{Nek}}$ and finally, by~(\ref{eqn:infofrom32}) and (\ref{eqn:willbeuseful}) to $\mathcal{L}_\xi^{\prime}(\phi_\alpha)$. This completes the proof.
\end{proof}

Next, we relate the right hand side of the statement of Theorem~\ref{main-primitive} to a special value of a $p$-adic $L$-function (that we call $L_{\xi,\Phi}$) which we construct below.

Let $\Phi_{\infty}$ denote the cyclotomic $\ZZ_p$-extension of $\QQ_p:=\Phi_0$, and let $\Phi_{n}$ denote the unique sub-extension of $\Phi_{\infty}/\QQ_p$ of degree $p^n$. Recall that $T^*=\textup{Hom}(T,\oo(1))\cong\oo(\chi)$. We set 
$$H^1_{\infty}(\QQ_p,T^*)=\varprojlim_n H^1(\Phi_{n},T^*),$$
where the inverse limit is taken with respect to norm maps. We may identify $\textup{Gal}(\Phi_\infty/\QQ_p)$ naturally by $\Gamma=\textup{Gal}(\QQ_\infty/\QQ)$. Let $\gamma$ be a topological generator for $\Gamma$ and let $\Lambda=\oo[[\Gamma]]$ as usual. 

\begin{lemma}
\label{lifting}
The natural map $H^1_{\infty}(\QQ_p,T^*) \ra H^1(\QQ_p,T^*)$ is surjective.
\end{lemma}

\begin{proof}
By \cite[Proposition II.1.1]{colmez-reciprocity}, we have $H^1_{\infty}(\QQ_p,T^*)\cong H^1(\QQ_p, T^*\otimes\Lambda)$ and the map above is simply the reduction map modulo $\gamma-1$. Hence, the cokernel of this map is $H^2(\QQ_p,T^*\otimes\LL)[\gamma-1]$, the $\gamma-1$ torsion of $H^2(\QQ_p,T^*\otimes\LL)$. Since the cohomological dimension of $\textup{Gal}(\overline{\QQ}_p/\QQ_p)$ is 2, it follows that 
$$H^2(\QQ_p,T^*\otimes\LL)/(\gamma-1)\cong H^2(\QQ_p,T^*\otimes\LL/(\gamma-1))=H^2(\QQ_p,T^*),$$
 which  is trivial (by local duality). Thus we have an exact sequence 
 $$\xymatrix{
0\ar[r]& H^2(\QQ_p,T^*\otimes\LL)[\gamma-1]\ar[r] &H^2(\QQ_p,T^*\otimes\LL) \ar[r]^(0.5){\gamma-1}& H^2(\QQ_p,T^*\otimes\LL) \ar[r]&0.}$$
It is know that $H^2(\QQ_p,T^*\otimes\LL)$ is an $\oo$-module of finite type (c.f., \cite[ Proposition 3.2.1]{pr}), thus it follows from~\cite[Theorem 2.4]{matsumura} that $H^2(\QQ_p,T^*\otimes\LL)[\gamma-1]=0$ as well, hence the lemma is proved. 
\end{proof}

By Lemma~\ref{lifting}, it is possible to choose $\Phi=\{\phi_\alpha^{(n)}\}_{n\geq 0} \in H^1_{\infty}(\QQ_p,T^*)$ such that $\phi_\alpha^{(0)}=\phi_\alpha$.

\begin{define}
\label{measure}
Attached to $\xi$ and $\Phi$, define an $\oo$-valued measure $\mu_{\xi, \Phi}$ on $\Gamma$ as follows: For $\tau \in \Gamma$, set 
$$\mu_{\xi,\Phi}(\tau\Gamma^{p^n})=\mathcal{L}_\xi(\tau\phi_\alpha^{(n)}).$$ 
\end{define}
The fact that $\mu_{\xi,\Phi}$ is a distribution follows from the fact that the collection $\{\Phi_\alpha^{(n)}\}_{n}$ is norm-compatible.

We  define the ``$p$-adic $L$-function'' associated to $\xi$ and $\Phi$ by setting 
$$L_{\xi,\Phi}(\eta)=\int_\Gamma\eta \,d\mu_{\xi,\Phi}$$
for each character $\eta:\Gamma \ra \ZZ_p^\times$. Let $\pmb{1}$ be the trivial character, and $\rho_{\textup{cyc}}:\Gamma\ra 1+p\ZZ_p$ be the cyclotomic character. We define the \emph{derivative at the trivial character} $\pmb{1}$ as 
$$L_{\xi,\Phi}^\prime(\pmb{1}):=\frac{d}{ds}L_{\xi,\Phi}(\rho_{\textup{cyc}}^s)\Big{|}_{s=0}.$$
We also define $\frak{P}_{\xi,\Phi}\in \LL$ to be the power series associated with the measure $\mu_{\xi,\Phi}$.
\begin{rem}
\label{rem:observationabuse}
Define 
$$P_n(\mu_{\xi,\Phi}):=\sum_{\tau\in \Gamma/\Gamma^{p^n}}\mu_{\xi,\Phi}(\tau\Gamma^{p^n})\cdot\tau\in \oo[\Gamma/\Gamma^{p^n}],$$
so that $\frak{P}_{\xi,\Phi}=\lim_n P_n(\mu_{\xi,\Phi}) \in \oo[[\Gamma]].$ For the powers  $\rho_{\textup{cyc}}^s:\Gamma \ra 1+p\ZZ_p$ of the cyclotomic character, observe that 
\be\label{eqn:usefulidentity}
\rho_{\textup{cyc}}^s(\frak{P}_{\xi,\Phi})=\lim_{n\ra\infty} \sum_{\tau\in \Gamma/\Gamma^{p^n}}\mu_{\xi,\Phi}(\tau\Gamma^{p^n})\cdot\rho_{\textup{cyc}}^s(\widetilde{\tau}).
\ee
Here, $\widetilde{\tau} \in \Gamma$ stands for an arbitrary lift of $\tau \in \Gamma/\Gamma^{p^n}$, and it is not hard  to see that the limit above does not depend on the choice of these lifts although  each  sum does depend on this choice. We therefore see that 
$\rho_{\textup{cyc}}^s(\frak{P}_{\xi,\Phi})=L_{\xi,\Phi}(\rho_{\textup{cyc}}^s)$, which in turn implies that 
$$\frac{d}{ds}\rho_{\textup{cyc}}^s(\frak{P}_{\xi,\Phi})\Big{|}_{s=0}=L_{\xi,\Phi}^\prime(\pmb{1}).$$
\end{rem}
\begin{prop}
\label{prop:main}
$
\mathcal{L}_{\xi}^{\prime}(\phi_\alpha)=L_{\xi,\Phi}^{\prime}(\pmb{1})$.
\end{prop}

\begin{rem}
\label{rem:dependance}
Note that the left hand side of the equality in Proposition~\ref{prop:main} depends only on  $\phi_\alpha$, not on its lift $\Phi$; whereas the right hand side depends  a priori on $\Phi$. Hence Proposition~\ref{prop:main} shows in particular that $L_{\xi,\Phi}^{\prime}(\pmb{1})$ does only depend on $\phi_\alpha$, and not on the lifting $\Phi$.
\end{rem}

\begin{cor}
\label{main}
$\langle c_1^\chi, \alpha\rangle_{\textup{Nek}}=
L_{\xi,\Phi}^{\prime}(\pmb{1})$.
\end{cor}

The proof of Proposition~\ref{prop:main} will be completed in a few steps, all of which are essentially borrowed from~\cite{ru94} with minor alterations. 

\begin{define}
\label{alternative-derivative}
Suppose $\mu=\mu^{(0)} \in H^1(\Phi_0,T^*)$ and 
$\mu=\{\mu^{(n)}\} \in \varprojlim H^1(\Phi_n,T^*)$. Define 
$$\textup{Der}_{\rho_{cyc}}(\mathcal{L}_{\xi})(\mu):=\lim_{n\ra\infty} \sum_{\tau\in \textup{Gal}(\QQ_n/\QQ)}\log_p(\rho_{\textup{cyc}}(\tau))\cdot \mathcal{L}_\xi(\tau\mu^{(n)}).$$ 
As the notation suggests, this definition only depends only on $\mu$, not on the lift $\mu$. This fact will follow from Lemma~\ref{lem3.1} below (where we also prove that the limit above exists). 
\end{define}

\begin{lemma}
\label{lem3.1}
Suppose $\nu \in H^1(\Phi_n,T^*)$ is such that $\mathbf{N}_{\Phi_n/\Phi_0}(\nu)=0$. Then 
$$\sum_{\tau\in \textup{Gal}(\QQ_n/\QQ)}\log_p(\rho_{\textup{cyc}}(\tau))\cdot \mathcal{L}_\xi(\tau\nu) \equiv 0 \mod p^n.$$
\end{lemma}
\begin{proof}
Fix $n$ and to ease notation, set $\mathcal{L}=\mathcal{L}_\xi \Big{|}_{H^1(\Phi_n,T^*)} \in \textup{Hom}\left(H^1(\Phi_n,T^*),\oo\right)$ and $G=\textup{Gal}(\QQ_n/\QQ)$. Write 
$$\delta =\sum_{\tau \in G} \log_p\left(\rho_{\textup{cyc}}(\tau)\right)\cdot\tau^{-1} \in \ZZ/p^n\ZZ[G].
$$
 Note that the claim of the Lemma is equivalent to showing that 
 \be\label{eqn:reduction} \delta\mathcal{L}(\nu)=0 \,\,(\hbox{in } \oo/p^n\oo ).
 \ee 
It is easy to see that 
\begin{align*}
 (\sigma-1)\delta&=\log_p\left( \rho_{\textup{cyc}}(\sigma)\right)\sum_{\tau \in G} \tau \\
&=\log_p(\rho_{\textup{cyc}}(\sigma))\cdot\textbf{N}_{\Phi_n/\Phi_0} \hbox{ , for all } \sigma \in G,
 \end{align*}
hence it follows that 
$$(\sigma-1)\delta\al=\log_p(\rho_{\textup{cyc}}(\sigma))\cdot\textbf{N}_{\Phi_n/\Phi_0}\al=0,$$
 where we have the final equality because $\al\big{|}_{H^1(\Phi_0,T^*)}=0$ by Lemma~\ref{lem:vanishing}. This is equivalent to saying that 
\be
\label{eqn:invariance of al}
\delta\al \in \textup{Hom}(H^1(\Phi_n,T^*),\oo/p^n\oo)^G.
\ee
Consider the map 
$$\xymatrix{ \textbf{N}^*: \textup{Hom}(H^1(\Phi_0,T^*),\oo/p^n\oo)\ar[rr]^(.51){-\,\,\circ\,\,\textbf{N}_{\Phi_n/\Phi_0}}&&  \textup{Hom}(H^1(\Phi_n,T^*),\oo/p^n\oo)^G.
}$$ 
Note that both of the sides of above are finite and the map $\textbf{N}^*$ is injective by Lemma~\ref{lifting}.  Claim below proves that there is an isomorphism 
$$ \textup{Hom}(H^1(\Phi_n,T^*),\oo/p^n\oo)^G\cong \textup{Hom}(H^1(\Phi_0,T^*),\oo/p^n\oo)$$
 which in turn implies that $\textbf{N}^*$ is surjective as well:
\begin{claim}
 \label{claim}
$ \textup{Hom}(H^1(\Phi_n,T^*),\oo/p^n\oo)^G\cong \textup{Hom}(H^1(\Phi_0,T^*),\oo/p^n\oo).$
\end{claim}
\begin{proof}[Proof of the Claim:]
By slight abuse, we let $\gamma$ denote a generator of $G$. Then, an element $\psi \in  \textup{Hom}(H^1(\Phi_n,T^*),\oo/p^n\oo)$ is fixed by $G$ if and only if 
\begin{align*}
\gamma^{-1}\psi=\psi &\iff \psi(\gamma x)=\psi(x) \hbox{ for all } x \in H^1(\Phi_n,T^*) \\
&\iff \psi((\gamma-1)x)=0 \hbox{ for all } x \in H^1(\Phi_n,T^*)\\
&\iff \psi \hbox{ factors through } H^1(\Phi_n,T^*)/(\gamma-1) \cong H^1(\Phi_0,T^*).
\end{align*}
where the very last isomorphism comes from the proof of Lemma~\ref{lifting}.
\end{proof}
We are now ready to complete the proof of Lemma~\ref{lem3.1}. It follows from our conclusion that $\textbf{N}^*$ is surjective that there exists  $g \in \textup{Hom} (H^1(\Phi_0,T^*),\oo/p^n\oo)$ such that $\delta\al =g\circ \textbf{N}_{\Phi_n/\Phi}$, hence
 $$\delta\al(\nu)=g(\textbf{N}_{\Phi_n/\Phi_0}(\nu))=0 \hbox{ in } \oo/p^n\oo.$$
  This is exactly the statement of (\ref{eqn:reduction}).

\end{proof}
\begin{rem}
\label{rem:compareprime}
As in the remark following Lemma 3.1 of~\cite{ru94}, one can check that
$$\textup{Der}_{\rho_{cyc}}(\mathcal{L}_{\xi})=
\mathcal{L}_\xi^\prime$$ 
using the fact that $H^1(\QQ_p,T\otimes\LL)$ has no $(\gamma-1)$-torsion. Here the equality takes place in $\textup{Hom}\left(H^1(\Phi_0,T^*),\oo\right)$. Note that the term involving the $p$-adic logarithm in loc.cit. does not appear here because we have already normalized $z_\infty^{\chi}$ by the factor $\log_p\rho_{\textup{cyc}}(\gamma)$.  
\end{rem}

\begin{proof}[Proof of Proposition~\ref{prop:main}] (Compare to~\cite[Proposition 7.1(ii)]{ru94}) By Remark~\ref{rem:compareprime}, 
\begin{align*}
\al_\xi^\prime(\phi_\alpha)&=\lim_{n\ra\infty}\sum_{\tau \in \textup{Gal}(\QQ_n/\QQ)} \log_p\rho_{\textup{cyc}}(\tau)\cdot \al_\xi(\tau\Phi^{(n)}_\alpha)\\
&=\lim_{n\ra\infty}\sum_{\tau \in \textup{Gal}(\QQ_n/\QQ)} \log_p\rho_{\textup{cyc}}(\tau)\mu_{\xi,\Phi}(\tau\Gamma^{p^n})\\
&=\int_\Gamma \log_p\rho_{\textup{cyc}}\cdot d\mu_{\xi,\Phi}.
\end{align*}
On the other hand $$\frac{d}{ds}\rho_{\textup{cyc}}^s=(\log_p\rho_{\textup{cyc}})\rho_{\textup{cyc}}^s,$$
hence 
\begin{align*}
L_{\xi,\Phi}^{\prime}(\pmb{1})=\frac{d}{ds}\left(\int_\Gamma \rho_{\textup{cyc}}^s\cdot d\mu_{\xi,\Phi}\right)\Bigg|_{s=0}&=\left(\int_\Gamma (\log_p\rho_{\textup{cyc}})\rho_{\textup{cyc}}^s\cdot d\mu_{\xi,\Phi}\right)\Bigg|_{s=0}\\
&=\int_\Gamma \log_p\rho_{\textup{cyc}}\cdot d\mu_{\xi,\Phi} \\
&=
\al_\xi^\prime(\phi_\alpha).
\end{align*}
\end{proof}

\section{$p$-adic $L$-functions and Nekov\'a\v{r}'s height pairing}
\label{sec:KLNek}
In this section, we obtain a formula for the leading term of  an \emph{imprimitive Kubota-Leopoldt $p$-adic $L$-function} in terms of Nekov\'{a}\v{r}'s height pairing, much in the spirit of a $p$-adic Gross-Zagier formula, using the Rubin-style formula we proved above. This in particular suggests a new interpretation of the classical $p$-adic Kronecker limit formula (c.f., \cite[Theorem 5.18]{w}, \cite[\S2.5]{deshalit}) and the formula of Ferrero-Greenberg~\cite{ferrerogreenberg}. 
\subsection{$p$-adic $L$-functions}
In this section, we give an overview of the well-known construction of the Kubota-Leopoldt $p$-adic $L$-function (resp., Katz's two variable $p$-adic $L$-function) using cyclotomic units (resp., elliptic units).
\subsubsection{Cyclotomic units and the Kubota-Leopoldt $p$-adic $L$-function}
\label{subsubsec:cycloprep}
Let $\omega: G_\QQ \ra (\ZZ_p^{\times})_{\textup{tors}}$ denote the Teichm\"uller character giving the action of $G_\QQ$ on the $p$-th roots of unity $\mu_p$. Fix an embedding $\oo \hookrightarrow \overline{\QQ}_p \hookrightarrow \mathbb{C}$ so that one can identify complex and $p$-adic characters of finite order of $G_{\QQ}$. Via this identification, a character $\rho$ of $\Gamma$ of finite order naturally extends to an $\oo$-algebra homomorphism $\rho: \LL \ra \overline{\QQ}_p$. 

For a character $\rho:G_\QQ\ra \oo\hookrightarrow \mathbb{C}$ of finite order, let $L(s,\rho)$ denote the associated Dirichlet $L$-series .
\begin{define}\label{def:KLpadic}
Attached to a non-trivial even Dirichlet character $\psi$ of $G_\QQ$ whose order is prime to $p$, there is an element $\al_\psi \in \LL$ such that for every $k\geq 1$ and every character $\rho$ of finite order of $\Gamma$, 
$$\rho_{\textup{cyc}}^k\rho (\al_\psi)=(1-\omega^{-k}\rho\psi(p)p^{k-1})L(1-k, \omega^{-k}\rho\psi).$$ 
See~\cite[Theorem 7.10]{w}. The element $\al_\psi$ is called the $p$-adic $L$-function attached to $\psi$.
\end{define}
\begin{rem}
\label{rem:KLsvariable}
Starting from $\al_\psi$ above, one may construct a function $L_p(s,\psi)$ (which is analytic at all $s\in \ZZ_p$) 
 by setting 
$$L_p(s,\psi)=\rho_
{\textup{cyc}}^{1-s}(\al_\psi).$$
\end{rem}
Recall that $L_n=L\QQ_n$ and $L_\infty=L\QQ_\infty$. For a prime $\frak{p}$, let $U_{n,\frak{p}}$ denote the local units inside $(L_n)_{_{\frak{p}}}$. Let $\mathcal{U}_{n}:=\prod_{\frak{p}|p}U_{n,\frak{p}}$ be the group of semi-local units and let $\mathcal{V}_n=\left(L_n\otimes\QQ_p\right)^\times=\prod_{\frak{p}|p}(L_n)_{\frak{p}}^\times$. By Kummer theory, we have an identification
\be
\label{eqn:kummerid}
H^1((L_n)_{_p},\oo(1)) \stackrel{\sim}{\lra} \widehat{\mathcal{V}_n} \,\,\,\,\hbox{    and     }\,\,\,\, H^1((\QQ_n)_{_p},T) \stackrel{\sim}{\lra} \mathcal{V}_n^{\chi}
 \ee
where we recall that $\widehat{A}$ denotes the $p$-adic completion of an abelian group $A$ and when $A$ is endowed with an action of $\textup{Gal}(L/\QQ)$, we write $A^{\chi}$ for the $\chi$-part of $\widehat{A}$. Define $\mathcal{U}_\infty=\varprojlim_{n} \mathcal{U}_n$ and $\mathcal{V}_\infty=\varprojlim_n \mathcal{V}_n$, where the inverse limits are taken with respect to the norm maps. The identifications (\ref{eqn:kummerid}) above  then gives in the limit
\be
\label{eqn:kummeridinfty}
H^1(\QQ_p,T\otimes\LL) \stackrel{\sim}{\lra} \mathcal{V}_\infty^\chi.
\ee

Coleman introduced in~\cite{coleman} a useful tool which as an input takes a norm coherent sequences in a tower of local fields and gives as an output a power series. More precisely, Coleman defines a $\LL$-module homomorphism
\be
\label{eqn:coleamanmap}
\frak{col}^\psi_\infty : \,\,\mathcal{U}_\infty^\psi \lra \oo[[\Gamma]] 
\ee
with the property that 
\be
\label{eqn:colemanmaponcylounits} 
\frak{col}^\psi_\infty(\xi_\infty^\psi) = \al_\psi,
\ee 	
where we recall that $\xi_\infty^\psi \in \mathcal{U}_\infty^\psi$ is the norm coherent sequence of cyclotomic units along the tower of fields $\{L_n\}_{n\geq0}$. Let $\gamma$ be a topological generator of $\Gamma$ as fixed above. If the character $\psi$ is unramified at $p$, then $\frak{col}^\psi_\infty$ extends uniquely to a homomorphism
\be\label{eqn:extendedcoleman}
\frak{col}^\psi_\infty : \,\,\mathcal{V}_\infty^\psi \lra \frac{1}{\gamma-1}\oo[[\Gamma]].
\ee
See~\cite[\S3]{sol-wild}, \cite[\S2]{greithermain} and  \cite[\S4]{Tsuji} for a detailed description of Coleman's map. 

We define using (\ref{eqn:extendedcoleman})
\be\label{eqn:extendedcolemannormalized}
\xymatrix@C=.2in{
\widetilde{\frak{col}}_\infty^\psi=\frac{\gamma-1}{\frac{1}{p}\log_p(\rho_{\textup{cyc}}(\gamma))}\times\frak{col}_\infty^\psi\,\,\,: \,\mathcal{V}_\infty^\psi\ar[rr]&&\LL,
}
\ee
so that 
\be\label{eqn:colemanoncyclo}
\widetilde{\frak{col}}_\infty^\psi(\xi_\infty^\psi)=\frac{\gamma-1}{\frac{1}{p}\log_p(\rho_{\textup{cyc}}(\gamma))} \times \al_\psi \,\,\,\,\,\,\,\hbox{    and   }\,\,\,\,\,\,\, \widetilde{\frak{col}}_\infty^\psi(z_\infty^\psi)=
p\al_\psi,
\ee
where $z_\infty^\psi\in \mathcal{V}_\infty^\psi$ is the collection of wild cyclotomic $p$-units. Note that $\frac{1}{p}\log_p(\rho_{\textup{cyc}}(\gamma)) \in \ZZ_p^\times$ since $\gamma \in \Gamma$ assumed to be a topological generator and since we assumed $p$ is odd.
\subsubsection{Elliptic units and Katz's $p$-adic $L$-function}
\label{subsubsec:Katz}
Let $\frak{O}$ be the completion of the  ring of integers of the maximal unramified extension of $\frak{F}$ and let $k$ be a quadratic imaginary number field such that $p$ splits in $k/\QQ$. Write $p=\wp\wp^*$ with $\wp\neq\wp^*$. We adapt the notation and hypotheses from~\S\ref{subsec:notation}, in particular, $k_\infty$ is the unique $\ZZ_p$-extension of $k$ which is unramified outside $\wp$ and $\Gamma=\textup{Gal}(k_\infty/k)$. Write $k(\frak{f}\wp^\infty)=\bigcup_{n\geq0} k(\frak{f}\wp^{n+1})$ and let  
$$\rho_{E}:\textup{Gal}\left(k(\frak{f}\wp^\infty)/k(\frak{f})\right)\lra \ZZ_p^\times$$
 be the character whose construction is sketched in \S\ref{subsec:notation};  and let $\rho_\Gamma$ be its restriction to $\Gamma$. We may similarly define $\rho_E^*$, $\Gamma^*$ and $\rho_{\Gamma^*}$ by replacing $\wp$ by $\wp^*$. 
 Set $\mathcal{G}=\textup{Gal}(k(\frak{f}p^\infty)/k(\frak{f}))$ and $\pmb{\LL}=\frak{O}[[\mathcal{G}]]$. We denote the Grossencharacter character attached to the elliptic curve $E$ also by $\rho_E$, which should cause no confusion since these two characters are related in a manner described in~\cite{Weil}. 
 
For a Grossencharacter $\psi$ of $k$ of type $A_0$ (in the sense of~\cite[\S II.1]{deshalit}) and an integral ideal $\frak{m}\subset k$, the Hecke $L$-series of $\psi$ (with modulus $\frak{m}$) is the complex valued function $L_{\infty,\frak{m}}(\psi,s)=\sum\psi(\frak{a})\mathbf{N}\frak{a}^{-s}$, where $\frak{a}$ runs over all integral ideals relatively prime to $\frak{m}$.  Let $d_k\in \ZZ^-$ be the discriminant of $K$. As before, let $\chi: G_k\ra \frak{O}^\times$ be a Dirichlet character whose order is prime to $p$ and let $\Omega$ be the positive real period of a global minimal model of $E$. For notational simplicity, write $\rho=\rho_E$ and $\rho^*=\rho_E^*$.

The following theorem describes the 2-variable $\wp$-adic $L$-function, first constructed by Katz \cite{katztwo} and Manin and Vishik.
\begin{thm}
\label{thm:katzpadicLfunc}
For $j,k \in \ZZ$, set  $\epsilon=\rho_E^k{\rho_E^*}^{j}\chi$. There is a $\wp$-adic period $\Omega_\wp \in \pmb{\LL}$ and an element $\al_\chi \in \pmb{\LL}$ such that for $0\leq-j<k$,
$$\Omega_\wp^{j-k}\al_\chi(\rho^k{{\rho}^*}^{j})=\Omega^{j-k}(k-1)!\left(\frac{\sqrt{-d_k}}{2\pi}\right)^j\cdot G(\epsilon)\left(1-\frac{\epsilon(\wp)}{p}\right)\cdot L_{\infty,\frak{\wp}}(\epsilon^{-1},0).$$  
\end{thm}
See~\cite[Theorem II.4.14]{deshalit} for details (e.g., for a definition of $G(\epsilon)$) and for the proof.

In this paper, we are only interested in the restriction $\al_{\chi}\big{|}_{\Gamma}$ of the 2-variable $p$-adic $L$-function $\al_\chi$ to characters of $\Gamma$. Starting from the one-variable $p$-adic $L$-function $\al_{\chi}\big{|}_{\Gamma}$, we define $\frak{L}_{\wp}(s,\chi)=\al_{\chi}\Big{|}_{\Gamma}(\rho_\Gamma^{1-s})$.

For $k_n$ as in \S\ref{subsec:notation}, write $L_n=L k_n$. For a prime $\frak{q}$, let $U_{n,\frak{q}}$ be the local units inside $(L_n)_{\frak{q}}$, and let $\mathcal{U}_n=\prod_{\frak{q}|\wp} U_{n,\frak{q}}$ be the group of semi-local units. Set $\mathcal{U}_\infty=\varprojlim_n \mathcal{U}_n$. As in \S\ref{subsubsec:cycloprep}, we consider Coleman's map
$$\frak{col}_\infty^\chi: \mathcal{U}_\infty^\chi{\otimes}_{\oo}\frak{O} \lra \frak{O}[[\Gamma]],$$
see~\cite[\S I.3.5]{deshalit} for a definition of this map. The map $\frak{col}_\infty^\chi$ here is the map ``\,$i$\,'' of loc.cit. restricted to the $\chi$-parts and to the $\Gamma$-direction.

Let $\frak{w}_n\in L_n^\times$ be the elliptic unit denoted by $\xi_n$ by Bley~\cite[\S3]{bley-wild}. The collection $\frak{w}_{\infty}^\chi:=\{\frak{w}_n^\chi\}\in \mathcal{U}_\infty^\chi$ is called the collection of \emph{wild elliptic units along $\Gamma$}. As wild cyclotomic units recovers the Kubota-Leopoldt $p$-adic $L$-function, wild elliptic units along $\Gamma$ may be used to obtain the one-variable $p$-adic $L$-function:
\be
\label{eqn:ellipticcolemanimage}
\frak{col}_\infty^\chi(\frak{w}_{\infty}^\chi)=\mathcal{L}_\chi\big{|}_{\Gamma}. 
\ee
This fact has been first proved by Coates and Wiles~\cite{coateswilesone}. For the 2-variable version of~(\ref{eqn:ellipticcolemanimage}), see~\cite{yagertwo} and \cite[\S IV]{deshalit}.

\subsection{Height computations for the base field $\QQ$: The case $\chi$ is even}
\label{subsec:htQeven}
Let $\chi$ be an even Dirichlet character as before. Recall that $\Phi_n=(\QQ_n)_p$, and recall also the fixed place $v_0$ of $L$ which is induced from the embedding $\iota_p:\overline{\QQ}\hookrightarrow \overline{\QQ}_p$. Write $v_0$ for the unique place of $L_n$ which lies above $v_0$ and define $\frak{L}_n=(L_n)_{v_0}$. In this section, we construct a  particular collection
$$\Phi =\{\phi^{(n)}\}_n\in H^1(\QQ_p,T^*\otimes\LL)= \varprojlim_n H^1(\Phi_n,T^*)$$
 starting from $\widetilde{\frak{col}}_\infty^\chi$, which we use together with Corollary~\ref{main} to prove a formula for the leading term of an imprimitive Kubota-Leopoldt $p$-adic $L$-function. 
 
 As in~(\ref{eqn:defphialpha}), we have identifications 
 \begin{align*}
 H^1(\Phi_n,\oo(\chi))=\left(\bigoplus_{v|p}H^1((L_{n})_{v},\oo)\right)^{\chi^{-1}}\stackrel{\xi_{\chi^{-1}}}{\lra} H^1(\frak{L}_n,\oo)&= \textup{Hom}(G_{\frak{L}_n},\oo)\\
 &=\textup{Hom} (\widehat{\frak{L}_n^{\times}},\oo).
 \end{align*}
 Here the direct sum is over the places of $L$ which lie above $p$ with the convention that the unique place of $L_n$ above a place $v|p$ of $L$ is also denoted by $v$. Also, $\xi_{\chi^{-1}}$ is the projection to the $v_0$-coordinate and the final equality is obtained by local class field theory. Furthermore, as in~(\ref{eqn:compositum1}), we have identifications
 $$ 
\xymatrix@R=.1in{
H^1(\Phi_n,\oo(1)\otimes\chi^{-1})=\left(\bigoplus_{v|p}H^1((L_n)_v,\oo(1))\right)^\chi\ar[r]^(.64){\sim}_(.64){\xi_{\chi}}&H^1(\frak{L}_n,\oo(1))= \widehat{\frak{L}_n^{\times}}\otimes_{\ZZ_p}\oo},
$$
 which, put together with the identification above gives isomorphisms
 \be\label{eqn:isompreparation}
 \textup{Hom}\left(H^1(\Phi_n,T),\oo\right)\stackrel{\sim}{\lra} \textup{Hom} (\widehat{\frak{L}_n^{\times}},\oo) \stackrel{\sim}{\lra} H^1(\Phi_n,T^*).
 \ee
 Note that both  isomorphisms in (\ref{eqn:isompreparation}) depend on the choice of $v_0$, yet the compositum of them does not.
 
 Let $\mathfrak{U}H^1(\Phi_n,T) \subset H^1(\Phi_n,T)$ denote submodule of \emph{universal norms} inside of $H^1(\Phi_n,T)$, i.e., the image of the canonical $\LL$-module homomorphism
 $$H^1(\QQ_p,T\otimes\LL)=\varprojlim_mH^1(\Phi_m,T) \lra H^1(\Phi_n,T).$$
 The Coleman map $\widetilde{\frak{col}}_\infty^\chi: \varprojlim_mH^1(\Phi_m,T)=\varprojlim_m \mathcal{V}_m^\chi \lra \LL$ induces (since it is $\LL$-linear) a $\oo[\Gamma_n]$-module homomorphism 
 $$\widetilde{\frak{col}}_n^\chi: \mathfrak{U}H^1(\Phi_n,T) \lra \oo[\Gamma_n].$$ 
 For a finitely generated $\oo[\Gamma_n]$-module $M$, there is a canonical isomorphism
 $$\xymatrix@R=.25cm{
\frak{b}: \textup{Hom}_{\oo}(M,\oo) \ar[r]^(.47){\sim}& \textup{Hom}_{\oo[\Gamma_n]}(M,\oo[\Gamma_n])\\ 
 f \ar@{|->}[r]& \left(m \mapsto\sum_{g \in \Gamma_n} f(g^{-1}m)\cdot g\right)
 }$$
 (c.f., \cite[Proposition VI.3.4]{brown}). Using the isomorphism  $\frak{b}$ applied with $M=\mathfrak{U}H^1(\Phi_n,T)$, we define $\phi^{(n)}$ by requiring $\frak{b}(\phi^{(n)})=\widetilde{\frak{col}}_n^\chi$. 
  \begin{lemma}
 \label{lem:extendphis}
 The $\oo$-module $H^1(\Phi_n,T)/\frak{U}H^1(\Phi_n,T)\cong \textup{coker}\left(H^1(\QQ_p,T\otimes\LL)\ra H^1(\Phi_n,T)\right)$ is free of rank one.
 \end{lemma}
 \begin{proof}
 By the long exact sequence of $G_{\QQ_p}$-cohomology we have
 $$ \textup{coker}\left(H^1(\QQ_p,T\otimes\LL)\ra H^1(\Phi_n,T)\right)=H^2(\QQ_p,T\otimes\LL)[\gamma^{p^n}-1].$$
 By~\cite[Proposition II.1.1]{colmez-reciprocity} and by local duality, we have 
 \begin{align*}
 H^2(\QQ_p,T\otimes\LL)=\varprojlim_n H^2(\Phi_n,T)&=\varprojlim_n\textup{Hom} \left(H^0(\Phi_n,\frak{F}/\oo(\chi)),\frak{F}/\oo)\right)\\
 &=\textup{Hom} \left(\varinjlim_n H^0(\Phi_n,\frak{F}/\oo(\chi)),\frak{F}/\oo)\right) \cong \oo,
 \end{align*}
 which is free of rank one as an $\oo$-module.
 \end{proof}
 
 \begin{rem}
 \label{rem:chooseorthog}
 In this remark, we give a further study of the the universal norms $\frak{U}H^1(\Phi_n,T)$ inside $H^1(\Phi_n,T)$. For notational simplicity, we assume $\oo=\ZZ_p$; the general case may  be treated tensoring all our conclusions in this remark by $\oo$. Furthermore, since we assume $\chi(p)=1$ (i.e., $\chi\large{\mid}_{G_{\QQ_p}}=\mathbf{1}$), it suffices to study the universal norms $\frak{U}H^1(\Phi_n,\ZZ_p(1))$ inside $H^1(\Phi_n,\ZZ_p(1))$.
 
 \begin{itemize}
 \item[(i)] Let $\varpi_n \in \Phi_n^\times$ be a uniformizer which is chosen in a way that $\mathbf{N}_{\Phi_{n}/\Phi_m} (\varpi_n)=\varpi_m$ for every $n\geq m$. Let $\frak{U_n}$ the units of $\Phi_n$. Kummer theory gives an identification 
 $$H^1(\Phi_n,\ZZ_p(1))=\widehat{\Phi_n^\times}=\varpi_n^{\ZZ_p}\times\widehat{\frak{U}_n}.$$
 Since $p \in H^1(\QQ_p,\ZZ_p(1))=p^{\ZZ_p}\times \widehat{\ZZ_p^\times}$ is a universal norm, it follows from Lemma~\ref{lem:extendphis} that no local unit (i.e., an element of $\widehat{\ZZ_p^\times} \subset \widehat{\QQ_p^\times}$) besides $1$ is a universal norm, and we have $\frak{U}H^1(\QQ_p,\ZZ_p(1))=p^{\ZZ_p}$ under the identification above. Set $Y_0=\widehat{\ZZ_p^\times}$, so that we have a decomposition $H^1(\QQ_p,\ZZ_p(1))=\frak{U}H^1(\Phi_0,\ZZ_p(1))\times Y_0$ into rank-one $\ZZ_p$-modules. Note that we adopt here the multiplicative notation for these abelian groups.
 
 \item[(ii)] For every $n\geq m$, the restriction map
 $$\textup{res}_{\Phi_m/\Phi_n}: H^1(\Phi_m,\ZZ_p(1)) \lra H^1(\Phi_n,\ZZ_p(1))^{\textup{Gal}(\Phi_n/\Phi_m)}\hookrightarrow H^1(\Phi_n,\ZZ_p(1))$$ 
 is  simply the natural injection $\widehat{\Phi_m^\times} \hookrightarrow \widehat{\Phi_n^\times}$. When $m=0$, write $\textup{res}_n$ for $\textup{res}_{\Phi_n/\QQ_p}$.
 \item[(iii)] If $1\neq u\in\widehat{\ZZ_p^\times} \subset H^1(\QQ_p,\ZZ_p(1))$, then $\textup{res}_n(u)$ is not a universal norm. Indeed, if otherwise, $\mathbf{N}_{\Phi_{n}/\QQ_p}(\textup{res}_n(u))=u^{p^n} \in\widehat{\ZZ_p^{\times}}$ would then be a universal norm and hence $u^{p^n}=1$ by (i). Since $\widehat{\ZZ_p^\times}$ is torsion-free, it follows that $u=1$. Let 
 $$Y_n=\textup{im}\left(Y_0 \stackrel{\textup{res}_n}{\lra}H^1(\Phi_n,\ZZ_p(1))\right).$$
 \item[(iv)] The quotient $H^1(\Phi_n,\ZZ_p(1))/Y_n=\widehat{\Phi_n^\times}{\big /} \textup{im}(\widehat{\ZZ_p^\times}\hookrightarrow \widehat{\Phi_n^\times})$ is torsion-free. Indeed, if an element of the quotient $\widehat{\Phi_n^\times}{\big /} \textup{im}(\widehat{\ZZ_p^\times}\hookrightarrow \widehat{\Phi_n^\times})$ represented by $x \in \widehat{\Phi_n^\times} - \widehat{\ZZ_p^\times}$ is $p$-torsion, so that $x^p \in \widehat{\ZZ_p^\times}$, then we would have $\mu_p \subset{\Phi_n^\times}$, which is not true. Hence, $Y_n$ is a free rank-one direct summand of $H^1(\Phi_n,\ZZ_p(1))$.
 \item[(v)] By Lemma~\ref{lem:extendphis}, we have
 \be\label{eqn:univnormrank}
 \textup{rank}_{\ZZ_p} \,\frak{U}H^1(\Phi_n,\ZZ_p(1))=\textup{rank}_{\ZZ_p} H^1(\Phi_n,\ZZ_p(1))-1.
 \ee
 \end{itemize}
 Using (iii), (iv) and (\ref{eqn:univnormrank}), we conclude that $H^1(\Phi_n,\ZZ_p(1))=\frak{U}H^1(\Phi_n,\ZZ_p(1))\times Y_n$ as $\ZZ_p$-modules.
 \end{rem}
   Remark~\ref{rem:chooseorthog}(v)  ensures that one may extend $\phi^{(n)}: \mathfrak{U}H^1(\Phi_n,T) \ra \oo$ to a homomorphism $H^1(\Phi_n,T)\ra \oo$, by declaring $\phi^{(n)}(c)=0$ for $c \in Y_n$. Note in particular for $n=0$ that the map $\phi^{(0)}=\widetilde{\frak{col}}_0^\chi:H^1(\QQ_p,T)\stackrel{\sim}{\ra}\widehat{L_{v_0}^\times}\otimes_{\ZZ_p}\oo \ra \oo$ (which is extended from $\frak{U}H^1(\QQ_p,T)$ as described above)  is unramified since it is identically \emph{zero} on the units $\widehat{\oo_{L_{v_0}}^\times}\otimes_{\ZZ_p}\oo$ by construction (as explained in Remark~\ref{rem:chooseorthog}(i)).

Let $\varpi_{v_0} \in L_{v_0}^\times$ be a uniformizer and set $\alpha(v_0)=\widetilde{\frak{col}}_0^\chi(\varpi_{v_0}) \in \oo$. Note that the value $\widetilde{\frak{col}}_0^\chi(\varpi_{v_0})$ is well defined thanks to the discussion in the preceding paragraph. Let $\frak{col}_0^\chi\in \widetilde{H}^1_f(\QQ,T^*)$ be the element which maps to $\alpha(v_0)$ under the  compositum of the isomorphisms (\ref{eqn:tilde identification}). Furthermore, one may verify without difficulty that the collection $\Phi=\{\phi^{(n)}\}$ chosen as in this section is norm-coherent and  the Rubin-style formula we proved (Corollary~\ref{main}) applies with the particular $\Phi$ we have constructed. Before stating the theorem we prove using these facts, we first define what we call the "imprimitive $p$-adic $L$-function". 
\begin{define}
\label{def:imprimitivepadicL}
For $\al_\chi \in \LL$ as above and for any topological generator $\gamma\in \Gamma$, write $\widetilde{\al_\chi}:=\frac{\gamma-1}{\frac{1}{p}\log_p\rho_{\textup{cyc}}(\gamma)}\times\al_\chi \in \LL,$ and define the \emph{imprimitive $p$-adic $L$-function} to be 
$$\widetilde{L}_p(s,\chi)=\rho_{\textup{cyc}}^{1-s}(\widetilde{\al_\chi}).$$
Note that,
\begin{itemize}
\item $\widetilde{L}_p(s,\chi)$ is an Iwasawa function,
\item $\frac{d}{ds}\widetilde{L}_p(s,\chi)\Big{|}_{s=1}$ does not depend on the choice of $\gamma$.
\end{itemize}
\end{define}
  \begin{thm}
\label{main:K-L}
Suppose $\chi(p)=1$ and let $\widetilde{L}_p(s,\chi)$ be the imprimitive $p$-adic $L$-function defined as above. Then 
$$\widetilde{L}_p^{\prime}(1,\chi)=\langle c_1^\chi,\frak{col}_0^\chi\rangle_{\textup{Nek}}.$$
\end{thm}
\begin{proof}
As in \S\ref{sec:main}, let $\mu_{\xi,\Phi}$ be the measure on $\Gamma$ attached to $\xi=\xi_\infty^\chi$ and $\Phi$ we chose as above, let $\frak{P}_{\xi,\Phi}\in \LL$ be the associated power series and let $L_{\xi,\Phi}(\eta)$ denote the `$p$-adic $L$-function' on the characters $\eta:\Gamma\ra\ZZ_p^\times$. We then have 
\begin{align*}
\frak{P}_{\xi,\Phi}=\widetilde{\frak{col}}_\infty^\chi(\xi_\infty^\chi)&=\frac{\gamma-1}{\frac{1}{p}\log_p\rho_{\textup{cyc}}(\gamma)}\times{\frak{col}}_{\infty}^\chi(\xi_\infty^\chi)\\
&=\frac{\gamma-1}{\frac{1}{p}\log_p\rho_{\textup{cyc}}(\gamma)}\times\al_\chi.
\end{align*}
We therefore see that
\be\label{eqn:measurecalc}
\frac{d}{ds}\rho_{\textup{cyc}}^{s}(\frak{P}_{\xi,\Phi})\Big{|}_{s=0}=p\cdot\pmb{1}(\al_{\chi})=p\cdot L_p(1,\chi)=\frac{d}{ds}\widetilde{L}_p(s,\chi)\Big{|}_{s=1},
\ee
where we have the first equality because $\frac{d}{ds}\rho_{\textup{cyc}}^s=\log_p\rho_{\textup{cyc}}\cdot \rho_{\textup{cyc}}^s$, the second thanks to our definition of $L_p(s,\chi)$ (see Remark~\ref{rem:KLsvariable}). 

On the other hand, we have $\frac{d}{ds}\rho_{\textup{cyc}}^{s}(\frak{P}_{\xi,\Phi})\Big{|}_{s=0}=L^{\prime}_{\xi,\Phi}(\pmb{1})$ by Remark~\ref{rem:observationabuse}, and the Theorem follows combining (\ref{eqn:measurecalc}) and Corollary~\ref{main}. 

\end{proof}
\begin{rem}
When $\chi$ is an even character with $\chi(p)=1$, the \emph{exceptionality} that Nekov\'a\v{r}'s extended Selmer groups detect are not due to an honest exceptional zero of the associated Kubota-Leopoldt $p$-adic $L$-function, but rather due to the fact that the extended Selmer groups correspond to an imprimitive $p$-adic $L$-function.
\end{rem}
\subsection{Height computations for the base field $\QQ$: The case $\chi$ is odd}
\label{subsec:htQodd}
We suppose now that $\chi:G_\QQ \ra \oo^\times$ is an odd Dirichlet character whose order is prime to $p$ and which has the property that $\chi(p)=1$. Keeping the notation of~\S\ref{subsec:tilde} and \S\ref{subsec:classical}, we have the following identifications as in Proposition~\ref{prop:classical explicit} and Corollary~\ref{cor:tilde explicit}:
\be\label{eqn:Texplicit}
 \widetilde{H}^1_{f}(\QQ,T)=H^1_{\FFc}(\QQ,T)=\left(\mathcal{O}_L\left[{1}/{p}\right]^\times\right)^{\chi}, 
\ee
\be\label{eqn:T*explicit}
H^0(\QQ_p,\oo(\chi)) \stackrel{\sim}{\lra}\widetilde{H}^1_{f}(\QQ,T^*).
\ee
In particular, $\widetilde{H}^1_{f}(\QQ,T^*)$ is a free $\oo$-module of rank one. Also, since $\chi$ is odd and $\chi(p)=1$, the $\oo$-module $\widetilde{H}^1_{f}(\QQ,T)$ is also free of rank one.

The assumption that $\chi(p)=1$ implies that the prime $p$ splits completely in $L/\QQ$. Let $\wp\subset L$ be any prime above $p$ and let $\iota_\wp : L \hookrightarrow L_\wp=\QQ_p$ be the induced embedding. Let $h$ denote the class number of $L$, and let $x \in \oo_L[{1}/{p}]^{\times}$ be such that $\oo_L\cdot x=\wp^h.$ Define 
\be\label{define:z}
z=e_\chi\cdot x \in (\oo_L[1/p]^\times)^\chi=\widetilde{H}^1_f(\QQ,T)\,\,\, \hbox{ and }\,\,\, z_0=\frac{1}{h}\cdot z \in \widetilde{H}^1_f(\QQ,T) \otimes\QQ_p.
\ee
It is not hard to see that the \emph{$\mathcal{L}$-invariant} (c.f., \cite[\S1]{gr})
$$\mathcal{L}:=\frac{\log_p(\iota_\wp(z))}{\textup{ord}_\wp(z)}= \log_p(\iota_\wp(z_0)) \in \frak{F}=\textup{Frac}(\oo)$$
is independent of the choice of the place $\wp$ and the choice of $x$.

Let $f=f_L$ be the conductor of the abelian field $L$. We regard the character $\chi$ as a character of the group $\Delta_f:=\textup{Gal}(\QQ(\mu_f)/\QQ)$ via 
$$\chi: \Delta_f \twoheadrightarrow \textup{Gal}(L/\QQ)\ra\oo^\times$$
 and  define the \emph{tame Stickelberger element}  
$$\theta_f=\sum_{a \in (\ZZ/f\ZZ)^\times\cong\Delta_f} \left(\frac{\langle a \rangle}{f}-\frac{1}{2}\right)\delta	_a^{-1}\in \oo[\Delta_f],$$
so that 
$$\chi(\theta_f)=B_{1,\chi^{-1}}=-L(0,\chi^{-1}),$$ where $B_{1,\chi^{-1}}$ is the generalized Bernoulli number.

Fixing generators $g_\chi$ of $\oo(\chi)$ and $g_{\chi^{-1}}$ of $\oo(\chi^{-1})$, and using the fact that $\chi(p)=1$, we obtain isomorphisms 
$$g_\chi: \,H^i(\QQ_p,T)\stackrel{\sim}{\ra}H^i(\QQ_p,\oo(1))\,\,\,\,\, \hbox{and}\,\,\,\,\,g_{\chi^{-1}}: \,H^i(\QQ_p,T^*)\stackrel{\sim}{\ra}H^i(\QQ_p,\oo)$$
for every $i\geq 0$. We choose $g_\chi$ and $g_{\chi^{-1}}$ so that the following diagram is commutative:
$$
\xymatrix@C=.2in{H^i(\QQ_p,T) \ar[d]_{g_\chi} &\otimes&H^{2-i}(\QQ_p,T^*) \ar[d]_{g_{\chi^{-1}}}\ar[rr]^(.65){\langle\,,\,\rangle_{\textup{Tate}}}  &&\oo\ar@{=}[d]\\
H^i(\QQ_p,\oo(1))&\otimes&H^{2-i}(\QQ_p,\oo)\ar[rr]^(.65){\langle\,,\,\rangle_{\textup{Tate}}}&& \oo
}$$
Via the identifications above, we view $\chi(\theta_f)$ as an element of $\widetilde{H}^1_f(\QQ,T^*)$. 

Let $\langle\,,\,\rangle_{\textup{Nek}}$ be Nekov\'a\v{r}'s height pairing as in~\S\ref{subsubsec:heights} above. We write $\langle\,,\,\rangle_{\textup{Nek}}$ also for the induced pairing
$$\left(\widetilde{H}^1_f(\QQ,T)\otimes\frak{F}\right)\otimes\left(\widetilde{H}^1_f(\QQ,T^*)\otimes\frak{F} \right)\stackrel{\langle\,,\,\rangle_{\textup{Nek}}}{\lra}\frak{F}.$$ 
\begin{thm}
\label{thm:htswhenodd}
$\langle z_0,\chi(\theta_f)\rangle_{\textup{Nek}}=-\al\cdot L(0,\chi^{-1})$.
\end{thm}

\begin{proof}
The statement of this Theorem is equivalent to the assertion that 
\be\label{eqn:reducetoz}\langle z,\chi(\theta_f)\rangle_{\textup{Nek}}=\log_{p}(\iota_\wp(z))\cdot\chi(\theta_f).
\ee
 As we have recalled in~\S\ref{subsubsec:heights}, we have $\langle z_0,\chi(\theta_f)\rangle_{\textup{Nek}}=\langle \beta^1(z_0),\chi(\theta_f)\rangle_{\textup{PT}}$, where 
 $$\beta^1: \widetilde{H}^1_f(\QQ,T)\lra \widetilde{H}^2_f(\QQ,T)\otimes\Gamma$$
  is the Bockstein map which is defined as follows:
 
 For $\frak{s} \in \widetilde{H}^1_f(\QQ,T)$, we define $\beta^1(\frak{s})=\frak{s} \cup \frak{c} \in \widetilde{H}^2_f(\QQ,T\otimes\Gamma)=\widetilde{H}^2_f(\QQ,T)\otimes\Gamma$, where $\frak{c}\in H^1(\QQ,\Gamma)=\textup{Hom}(G_\QQ,\Gamma)$ is the tautological homomorphism $\frak{c}:G_\QQ \ra \Gamma$. One similarly defines 
 $$\beta^1_p: H^1(\QQ_p,T) \lra H^2(\QQ_p,T)\otimes\Gamma$$
 by taking cup product with the element $\frak{c}_p \in H^1(\QQ_p,\Gamma)=\textup{Hom}(G_{\QQ_p},\Gamma)$, which is the restriction of $\frak{c}$ to $G_{\QQ_p}$. We then have the following commutative diagram:
 $$
 \xymatrix@C=.05in{
 \widetilde{H}^1_f(\QQ,T)\ar[d]\ar[rrrrr]^(.48){\beta^1}&&&&&\widetilde{H}^2_f(\QQ,T)\otimes\Gamma \ar[d]&\otimes& \widetilde{H}^1_f(\QQ,T^*)\ar[rrrrr]^(.65){\langle\,,\,\rangle_{\textup{PT}}}&&&&& \Gamma\ar@{=}[d]\ar[rrrr]^{\log_p\circ\,\rho_{\textup{cyc}}}&&&&\oo\\
 H^1(\QQ_p,T)\ar[rrrrr]_(.45){\beta_p^1}&&&&&H^2(\QQ_p,T)\otimes\Gamma&\otimes& H^0(\QQ_p,T^*)\ar[u]^{\cong}\ar[rrrrr]_(.65){\langle\,,\,\rangle_{\textup{Tate}}}&&&&& \Gamma\ar[rrrr]^{\log_p\circ\,\rho_{\textup{cyc}}}&&&&\oo
 }
 $$
 Here, the square on the left is commutative thanks to the description of $\beta^1$ and $\beta^1_p$ above, and the square on the right is commutative by the definition of the Poitou-Tate global pairing as the sum of local invariants, and thanks to the fact that $H^2(\QQ_\ell,T)=0$ for $\ell | \ff_\chi$. The proof of Theorem follows from the following Lemma, whose first part is a restatement of~\cite[11.3.5.3]{nek} and second part is~\cite[Lemma II.1.4.5]{katodr}:
 \begin{lemma}
 \begin{itemize} Suppose $\alpha \in H^1(\QQ_p,\oo(1))=\widehat{\QQ_p^\times}$, and suppose $\frak{a}_p:\widehat{\QQ_p^\times}\ra G_{\QQ_p}^{\textup{ab}}$ is the local reciprocity map as before.
 \item[(i)] $\textup{inv}_p(\beta^1_p(\alpha))=\textup{inv}_p(\alpha \cup \frak{c}_p)=\frak{c}_p(\frak{a}_p(\alpha))$.
 \item[(ii)] $\log_p\circ\,\rho_{\textup{cyc}}\circ\,\frak{c}_p\left(\frak{a}_p(\alpha)\right)=\log_p(\alpha)$.
 \end{itemize}
 \end{lemma}
\end{proof}

 \begin{rem}
 \label{rem:comparewithFG}
 The interpolation property that the $p$-adic $L$-function $L_p(s,\chi^{-1}\omega)$ satisfies (see Definition~\ref{def:KLpadic}), along with our assumption that $\chi(p)=1$ forces the Kubota-Leopoldt $p$-adic $L$-function  $L_p(0,\chi^{-1}\omega)$ to vanish at $s=0$. The theorem of Ferrero-Greenberg~\cite{ferrerogreenberg} combined with a result of Gross and Koblitz~\cite{grosskoblitz} shows that 
 $$\frac{d}{ds}L_p(s,\chi^{-1}\omega)\big{|}_{s=0}=-\al\cdot L(0,\chi^{-1}).$$
 Thus, Theorem~\ref{thm:htswhenodd} implies that
 \be\label{eqn:nekfg}\frac{d}{ds}L_p(s,\chi^{-1}\omega)\big{|}_{s=0}=\langle z_0,\chi(\theta_f)\rangle_{\textup{Nek}}.\ee
 This provides us with a new interpretation of the Ferrero-Greenberg theorem. Of course, it would be desirable to prove \emph{first} a Rubin-style formula (as we did in~\S\ref{sec:main}) in this setting and from that deduce (\ref{eqn:nekfg}) and the Ferrero-Greenberg theorem (as we prove a $p$-adic Kronecker formula from a Rubin-style formula in~\S\ref{subsec:htQeven} above and \S\ref{subsec:caseimagquad} below).
 \end{rem}
 
 \begin{rem}
 \label{rem:totallyrealDDP}
 Suppose in this remark that our base field $K$ is an arbitrary totally real number field and $\chi:G_K\ra\oo^\times$ is a totally odd character which has finite prime-to-$p$ order. Assume further that $\chi(\wp)=1$ for exactly one prime $\wp\subset K$ above $p$. In this setting, Gross conjectured in~\cite{grosspadic} a formula for the leading coefficient $L_p^\prime(0,\chi^{-1}\omega)$ of the Deligne-Ribet $p$-adic $L$-function $L_p(s,\chi^{-1}\omega)$ at $s=0$, and Darmon, Dasgupta and Pollack recently announced a proof of a portion of this conjecture. Using their result, we may express $L_p^\prime(0,\chi^{-1}\omega)$ in terms of of Nekov\'a\v{r}'s heights exactly as we did above for the Kubota-Leopoldt $p$-adic $L$-function when $K=\QQ$. 
 
 On the other hand, if one succeeds in proving a Rubin-style formula in this setting\footnote{It is expected that obtaining a Rubin-style formula for a general totally real $k$ (and for a totally odd character $\chi$) should not be any harder than proving such a formula for $k=\QQ$.}, then one in turn would obtain an alternative proof of Gross' conjecture.
 \end{rem}

\subsection{Height computations for a totally imaginary base field $k$}
\label{subsec:caseimagquad}
We keep the notation from \S\ref{subsubsec:Katz}. Every Dirichlet character $\chi$ of $G_k$ \emph{behaves} like an even character and the results we presented in \S\ref{sec:main} and \S\ref{subsec:htQeven} extend to this case without an extra effort. Replacing the cyclotomic units by elliptic units, and the results of~\cite{BG} by that of~\cite{bley-etnc}; the results of~\cite{sol-wild} by that of~\cite{bley-wild}, one may prove the following formula:
\begin{thm}
\label{thm:mainelliptic} Suppose $\chi(\wp)=1$. 
Then
$$\widetilde{\frak L}_p^\prime(1,\chi)=\langle \frak{e}_{1}^\chi,\frak{col}_0^\chi\rangle_{\textup{Nek}}.$$
\end{thm}
Here we follow the notation from~\S\ref{subsubsec:Katz}. Namely, 
\begin{itemize}
\item $\frak{e}_{1}$ is the (tame) elliptic unit which is denoted by $\mathbf{N}_{k(\frak{f})/L}\psi(1, \frak{f},\frak{a})$ in~\cite{bley-wild} and $\frak{e}_1^\chi \in \widetilde{H}^1_f(k,T)=H^1_{\FFc}(k,T)=L^{\times,\chi}$ is the $\chi$ part of $\frak{e}_{1}$,
\item $\frak{col}_0^\chi \in \widetilde{H}^1_f(k,T^*)$ is the element which is obtained from the Coleman map (as in \S\ref{subsec:htQeven}),
\item $\widetilde{\frak L}_p(s,\chi)=\frac{\rho_{\textup{cyc}}^{1-s}(\gamma)-1}{\frac{1}{p}\log_p\rho_{\textup{cyc}(\gamma)}}\cdot \frak{L}_p(s,\chi)$ is the imprimitive (one-variable) Katz $p$-adic $L$-function, where $\frak{L}_p(s,\chi)$ is the restriction of the two-variable $p$-adic $L$-function to $\Gamma$.
\end{itemize}
 
 \begin{rem}
 \label{rem:ellipticcurvesKurihara}
 Suppose $E_{/\QQ}$ is an elliptic curve and only in this remark, let $T=T_p(E)$ be the $p$-adic Tate-module of $E$. Let $L_p(E,s)$ denote the Mazur-Tate-Teitelbaum  $p$-adic $L$-function attached to $E$. Assume that $E$ has split multiplicative reduction at $p$. In this case, $L_p(E,s)$ has an exceptional zero at $s=1$ which is forced by the interpolation property. The Mazur-Tate-Teitelbaum conjecture (now a theorem of Greenberg and Stevens~\cite{grste}) asserts that 
 \be\label{eqn:mttconj}
 \frac{d}{ds}L_p(E,s)\big{|}_{s=1}=\al_E\cdot \frac{L(E,1)}{\Omega_E^+}
 \ee
 where $\al_E$ is the $\al$-invariant, $L(E,1)$ is the value of the Hasse-Weil $L$-function at $s=1$ and $\Omega_E^+$ is the real period of $E$. 
 
  Let 
  $$\langle\,,\,\rangle_{\textup{Tate}}:\,\, H^1(\QQ_p,T)\otimes H^1(\QQ_p,T^*)\lra \ZZ_p$$
  denote Tate's local cup-product pairing.
  M. Kurihara has kindly explained us how one may interpret the quantity on the right in (\ref{eqn:mttconj}) as the local Tate pairing calculated on Kato's zeta-element $\mathcal{Z}_0\in H^1(\QQ_p,T)$ and another special element $\alpha \in H^1(\QQ_p,T^*)$ (which we do not define here). Using this observation, Kurihara was able to give another proof of the Mazur-Tate-Teitelbaum conjecture (\ref{eqn:mttconj}). 
  
If one succeeds in proving a \emph{Rubin-style formula} in this setting, one could \emph{globalize} Kurihara's calculation with Kato's zeta-element $\mathcal{Z}_0$ and the element $\alpha$, so as to obtain a $p$-adic Gross-Zagier formula in the presence of an exceptional zero (i.e., relate Nekov\'a\v{r}'s height pairing to the Mazur-Tate-Teitelbaum $p$-adic $L$-function via \begin{enumerate} \item a \emph{Rubin-style formula} to relate heights to local Tate pairing, \item then using Kurihara's local calculation),\end{enumerate} much in the spirit of~\cite{bertdarmonexceptional1, bertdarmonexceptional2}.
 \end{rem}
\bibliographystyle{alpha}
\bibliography{heights}

\begin{thebibliography}{How05}

\bibitem[BD96]{bertdarmonexceptional1}
M.~Bertolini and H.~Darmon.
\newblock Heegner points on {M}umford-{T}ate curves.
\newblock {\em Invent. Math.}, 126(3):413--456, 1996.

\bibitem[BD97]{bertdarmonexceptional2}
Massimo Bertolini and Henri Darmon.
\newblock A rigid analytic {G}ross-{Z}agier formula and arithmetic
  applications.
\newblock {\em Ann. of Math. (2)}, 146(1):111--147, 1997.
\newblock With an appendix by Bas Edixhoven.

\bibitem[BG03]{BG}
D.~Burns and C.~Greither.
\newblock On the equivariant {T}amagawa number conjecture for {T}ate motives.
\newblock {\em Invent. Math.}, 153(2):303--359, 2003.

\bibitem[Ble04]{bley-wild}
W.~Bley.
\newblock Wild {E}uler systems of elliptic units and the equivariant {T}amagawa
  number conjecture.
\newblock {\em J. Reine Angew. Math.}, 577:117--146, 2004.

\bibitem[Ble06]{bley-etnc}
W.~Bley.
\newblock Equivariant {T}amagawa number conjecture for abelian extensions of a
  quadratic imaginary field.
\newblock {\em Doc. Math.}, 11:73--118 (electronic), 2006.

\bibitem[Bro94]{brown}
Kenneth~S. Brown.
\newblock {\em Cohomology of groups}, volume~87 of {\em Graduate Texts in
  Mathematics}.
\newblock Springer-Verlag, New York, 1994.
\newblock Corrected reprint of the 1982 original.

\bibitem[Col79]{coleman}
Robert~F. Coleman.
\newblock Division values in local fields.
\newblock {\em Invent. Math.}, 53(2):91--116, 1979.

\bibitem[Col98]{colmez-reciprocity}
Pierre Colmez.
\newblock Th\'eorie d'{I}wasawa des repr\'esentations de de {R}ham d'un corps
  local.
\newblock {\em Ann. of Math. (2)}, 148(2):485--571, 1998.

\bibitem[CW78]{coateswilesone}
J.~Coates and A.~Wiles.
\newblock On {$p$}-adic {$L$}-functions and elliptic units.
\newblock {\em J. Austral. Math. Soc. Ser. A}, 26(1):1--25, 1978.

\bibitem[dS87]{deshalit}
Ehud de~Shalit.
\newblock {\em Iwasawa theory of elliptic curves with complex multiplication},
  volume~3 of {\em Perspectives in Mathematics}.
\newblock Academic Press Inc., Boston, MA, 1987.
\newblock $p$-adic $L$ functions.

\bibitem[FG78]{ferrerogreenberg}
Bruce Ferrero and Ralph Greenberg.
\newblock {On the behavior of p-adic L-functions at s=0.}
\newblock {\em Invent. Math.}, 50:91--102, 1978.

\bibitem[GK79]{grosskoblitz}
Benedict~H. Gross and Neal Koblitz.
\newblock Gauss sums and the {$p$}-adic {$\Gamma $}-function.
\newblock {\em Ann. of Math. (2)}, 109(3):569--581, 1979.

\bibitem[Gre89]{g1}
Ralph Greenberg.
\newblock Iwasawa theory for {$p$}-adic representations.
\newblock In {\em Algebraic number theory}, volume~17 of {\em Adv. Stud. Pure
  Math.}, pages 97--137. Academic Press, Boston, MA, 1989.

\bibitem[Gre92]{greithermain}
Cornelius Greither.
\newblock Class groups of abelian fields, and the main conjecture.
\newblock {\em Ann. Inst. Fourier (Grenoble)}, 42(3):449--499, 1992.

\bibitem[Gre94]{gr}
Ralph Greenberg.
\newblock Trivial zeros of {$p$}-adic {$L$}-functions.
\newblock In {\em $p$-adic monodromy and the Birch and Swinnerton-Dyer
  conjecture (Boston, MA, 1991)}, volume 165 of {\em Contemp. Math.}, pages
  149--174. Amer. Math. Soc., Providence, RI, 1994.

\bibitem[Gro81]{grosspadic}
Benedict~H. Gross.
\newblock {$p$}-adic {$L$}-series at {$s=0$}.
\newblock {\em J. Fac. Sci. Univ. Tokyo Sect. IA Math.}, 28(3):979--994 (1982),
  1981.

\bibitem[GS93]{grste}
Ralph Greenberg and Glenn Stevens.
\newblock {$p$}-adic {$L$}-functions and {$p$}-adic periods of modular forms.
\newblock {\em Invent. Math.}, 111(2):407--447, 1993.

\bibitem[GZ86]{gz}
Benedict~H. Gross and Don~B. Zagier.
\newblock Heegner points and derivatives of {$L$}-series.
\newblock {\em Invent. Math.}, 84(2):225--320, 1986.

\bibitem[How04]{howard:heights}
Benjamin Howard.
\newblock Derived {$p$}-adic heights and {$p$}-adic {$L$}-functions.
\newblock {\em Amer. J. Math.}, 126(6):1315--1340, 2004.

\bibitem[How05]{howard:GZ}
Benjamin Howard.
\newblock The {I}wasawa theoretic {G}ross-{Z}agier theorem.
\newblock {\em Compos. Math.}, 141(4):811--846, 2005.

\bibitem[Kat76]{katztwo}
Nicholas~M. Katz.
\newblock {$p$}-adic interpolation of real analytic {E}isenstein series.
\newblock {\em Ann. of Math. (2)}, 104(3):459--571, 1976.

\bibitem[Kat93]{katodr}
Kazuya Kato.
\newblock Lectures on the approach to {I}wasawa theory for {H}asse-{W}eil
  {$L$}-functions via {$B\sb {\rm dR}$}. {I}.
\newblock In {\em Arithmetic algebraic geometry ({T}rento, 1991)}, volume 1553
  of {\em Lecture Notes in Math.}, pages 50--163. Springer, Berlin, 1993.

\bibitem[Kat04]{kato}
Kazuya Kato.
\newblock {$p$}-adic {H}odge theory and values of zeta functions of modular
  forms.
\newblock {\em Ast\'erisque}, (295):ix, 117--290, 2004.
\newblock Cohomologies $p$-adiques et applications arithm\'etiques. III.

\bibitem[Mat89]{matsumura}
Hideyuki Matsumura.
\newblock {\em Commutative ring theory}, volume~8 of {\em Cambridge Studies in
  Advanced Mathematics}.
\newblock Cambridge University Press, Cambridge, second edition, 1989.
\newblock Translated from the Japanese by M. Reid.

\bibitem[MR04]{mr02}
Barry Mazur and Karl Rubin.
\newblock Kolyvagin systems.
\newblock {\em Mem. Amer. Math. Soc.}, 168(799):viii+96, 2004.

\bibitem[Nek93]{nek2}
Jan Nekov{\'a}{\v{r}}.
\newblock On {$p$}-adic height pairings.
\newblock In {\em S\'eminaire de Th\'eorie des Nombres, Paris, 1990--91},
  volume 108 of {\em Progr. Math.}, pages 127--202. Birkh\"auser Boston,
  Boston, MA, 1993.

\bibitem[Nek95]{nek95}
Jan Nekov{\'a}{\v{r}}.
\newblock On the {$p$}-adic height of {H}eegner cycles.
\newblock {\em Math. Ann.}, 302(4):609--686, 1995.

\bibitem[Nek06]{nek}
Jan Nekov{\'a}{\v{r}}.
\newblock Selmer complexes.
\newblock {\em Ast\'erisque}, (310):viii+559, 2006.

\bibitem[PR87]{pr87}
Bernadette Perrin-Riou.
\newblock Points de {H}eegner et d\'eriv\'ees de fonctions {$L$} {$p$}-adiques.
\newblock {\em Invent. Math.}, 89(3):455--510, 1987.

\bibitem[PR92]{pr:ht}
Bernadette Perrin-Riou.
\newblock Th\'eorie d'{I}wasawa et hauteurs {$p$}-adiques.
\newblock {\em Invent. Math.}, 109(1):137--185, 1992.

\bibitem[PR93]{pr93}
Bernadette Perrin-Riou.
\newblock Fonctions {$L$} {$p$}-adiques d'une courbe elliptique et points
  rationnels.
\newblock {\em Ann. Inst. Fourier (Grenoble)}, 43(4):945--995, 1993.

\bibitem[PR94]{pr}
Bernadette Perrin-Riou.
\newblock Th\'eorie d'{I}wasawa des repr\'esentations {$p$}-adiques sur un
  corps local.
\newblock {\em Invent. Math.}, 115(1):81--161, 1994.
\newblock With an appendix by Jean-Marc Fontaine.

\bibitem[PR83]{pr83}
Bernadette Perrin-Riou.
\newblock Descente infinie et hauteur {$p$}-adique sur les courbes elliptiques
  \`a multiplication complexe.
\newblock {\em Invent. Math.}, 70(3):369--398, 1982/83.

\bibitem[Rub92]{ru:rational92}
Karl Rubin.
\newblock {$p$}-adic {$L$}-functions and rational points on elliptic curves
  with complex multiplication.
\newblock {\em Invent. Math.}, 107(2):323--350, 1992.

\bibitem[Rub94]{ru94}
Karl Rubin.
\newblock Abelian varieties, {$p$}-adic heights and derivatives.
\newblock In {\em Algebra and number theory (Essen, 1992)}, pages 247--266. de
  Gruyter, Berlin, 1994.

\bibitem[Rub00]{r00}
Karl Rubin.
\newblock {\em Euler systems}, volume 147 of {\em Annals of Mathematics
  Studies}.
\newblock Princeton University Press, Princeton, NJ, 2000.
\newblock Hermann Weyl Lectures. The Institute for Advanced Study.

\bibitem[Sol92]{sol-wild}
David Solomon.
\newblock On a construction of {$p$}-units in abelian fields.
\newblock {\em Invent. Math.}, 109(2):329--350, 1992.

\bibitem[Sol94]{sol94}
David Solomon.
\newblock {Galois relations for cyclotomic numbers and $p$-units.}
\newblock {\em J. Number Theory}, 46(2):158--178, 1994.

\bibitem[Tsu99]{Tsuji}
Takae Tsuji.
\newblock Semi-local units modulo cyclotomic units.
\newblock {\em J. Number Theory}, 78(1):1--26, 1999.

\bibitem[Was82]{w}
Lawrence~C. Washington.
\newblock {\em Introduction to cyclotomic fields}, volume~83 of {\em Graduate
  Texts in Mathematics}.
\newblock Springer-Verlag, New York, 1982.

\bibitem[Wei56]{Weil}
Andr{\'e} Weil.
\newblock On a certain type of characters of the id\`ele-class group of an
  algebraic number-field.
\newblock In {\em Proceedings of the international symposium on algebraic
  number theory, {T}okyo \& {N}ikko, 1955}, pages 1--7, Tokyo, 1956. Science
  Council of Japan.

\bibitem[Yag82]{yagertwo}
Rodney~I. Yager.
\newblock On two variable {$p$}-adic {$L$}-functions.
\newblock {\em Ann. of Math. (2)}, 115(2):411--449, 1982.

\end{thebibliography}
\end{document}